\documentclass[12pt]{article}
\usepackage{amsthm, amsmath, amssymb, enumerate}
\usepackage{fullpage}
\usepackage[colorlinks=true]{hyperref}
\usepackage[all]{xy}

\begin{document}
\title{Explicit Kodaira--Spencer map over Hilbert modular varieties}
\author{Ziqi Guo}
\maketitle

\theoremstyle{plain}
\newtheorem{thm}{Theorem}[section]
\newtheorem{theorem}[thm]{Theorem}
\newtheorem{cor}[thm]{Corollary}
\newtheorem{corollary}[thm]{Corollary}
\newtheorem{lem}[thm]{Lemma}
\newtheorem{lemma}[thm]{Lemma}
\newtheorem{pro}[thm]{Proposition}
\newtheorem{proposition}[thm]{Proposition}
\newtheorem{prop}[thm]{Proposition}
\newtheorem{definition}[thm]{Definition}
\newtheorem{assumption}[thm]{Assumption}

\theoremstyle{remark} 
\newtheorem{remark}[thm]{Remark}
\newtheorem{example}[thm]{Example}
\newtheorem{remarks}[thm]{Remarks}
\newtheorem{problem}[thm]{Problem}
\newtheorem{exercise}[thm]{Exercise}
\newtheorem{situation}[thm]{Situation}
\newtheorem{acknowledgment}[thm]{Acknowledgment}

\numberwithin{equation}{subsection}

\newcommand{\ZZ}{\mathbb{Z}}
\newcommand{\CC}{\mathbb{C}}
\newcommand{\QQ}{\mathbb{Q}}
\newcommand{\RR}{\mathbb{R}}
\newcommand{\HH}{\mathcal{H}}     

\newcommand{\ad}{\mathrm{ad}}            
\newcommand{\NT}{\mathrm{NT}}         
\newcommand{\nonsplit}{\mathrm{nonsplit}}         
\newcommand{\Pet}{\mathrm{Pet}}         
\newcommand{\Fal}{\mathrm{Fal}}         

\newcommand{\cs}{{\mathrm{cs}}}         

\newcommand{\ZZn}{\mathbb{Z}[\frac{1}{n}]}         
\newcommand{\ZZN}{\mathbb{Z}[\frac{1}{N}]}     
\newcommand{\XU}{X_U}    
\newcommand{\XXU}{\mathcal{X}_U}    
\newcommand{\OA}{\underline{\Omega}_\mathcal{A}}
\newcommand{\OU}{\Omega_{\mathcal{X}_U/\mathbb{Z}[\frac{1}{n}]}}
\newcommand{\WA}{\underline{\omega}_\mathcal{A}}
\newcommand{\WU}{\omega_{\mathcal{X}_U/\mathbb{Z}[\frac{1}{n}]}}
\newcommand{\HHom}{\mathcal{H}\mathrm{om}}

\newcommand{\pair}[1]{\langle {#1} \rangle}
\newcommand{\wpair}[1]{\left\{{#1}\right\}}
\newcommand{\wh}{\widehat}
\newcommand{\wt}{\widetilde}

\newcommand\Spf{\mathrm{Spf}}

\newcommand{\lra}{{\longrightarrow}}

\newcommand{\matrixx}[4]
{\left( \begin{array}{cc}
  #1 &  #2  \\
  #3 &  #4  \\
 \end{array}\right)}        


\newcommand{\BA}{{\mathbb {A}}}
\newcommand{\BB}{{\mathbb {B}}}
\newcommand{\BC}{{\mathbb {C}}}
\newcommand{\BD}{{\mathbb {D}}}
\newcommand{\BE}{{\mathbb {E}}}
\newcommand{\BF}{{\mathbb {F}}}
\newcommand{\BG}{{\mathbb {G}}}
\newcommand{\BH}{{\mathbb {H}}}
\newcommand{\BI}{{\mathbb {I}}}
\newcommand{\BJ}{{\mathbb {J}}}
\newcommand{\BK}{{\mathbb {K}}}
\newcommand{\BL}{{\mathbb {L}}}
\newcommand{\BM}{{\mathbb {M}}}
\newcommand{\BN}{{\mathbb {N}}}
\newcommand{\BO}{{\mathbb {O}}}
\newcommand{\BP}{{\mathbb {P}}}
\newcommand{\BQ}{{\mathbb {Q}}}
\newcommand{\BR}{{\mathbb {R}}}
\newcommand{\BS}{{\mathbb {S}}}
\newcommand{\BT}{{\mathbb {T}}}
\newcommand{\BU}{{\mathbb {U}}}
\newcommand{\BV}{{\mathbb {V}}}
\newcommand{\BW}{{\mathbb {W}}}
\newcommand{\BX}{{\mathbb {X}}}
\newcommand{\BY}{{\mathbb {Y}}}
\newcommand{\BZ}{{\mathbb {Z}}}

\newcommand{\CA}{{\mathcal {A}}}
\newcommand{\CB}{{\mathcal {B}}}
\newcommand{\CD}{{\mathcal{D}}}
\newcommand{\CE}{{\mathcal {E}}}
\newcommand{\CF}{{\mathcal {F}}}
\newcommand{\CG}{{\mathcal {G}}}
\newcommand{\CH}{{\mathcal {H}}}
\newcommand{\CI}{{\mathcal {I}}}
\newcommand{\CJ}{{\mathcal {J}}}
\newcommand{\CK}{{\mathcal {K}}}
\newcommand{\CL}{{\mathcal {L}}}
\newcommand{\CM}{{\mathcal {M}}}
\newcommand{\CN}{{\mathcal {N}}}
\newcommand{\CO}{{\mathcal {O}}}
\newcommand{\CP}{{\mathcal {P}}}
\newcommand{\CQ}{{\mathcal {Q}}}
\newcommand{\CR }{{\mathcal {R}}}
\newcommand{\CS}{{\mathcal {S}}}
\newcommand{\CT}{{\mathcal {T}}}
\newcommand{\CU}{{\mathcal {U}}}
\newcommand{\CV}{{\mathcal {V}}}
\newcommand{\CW}{{\mathcal {W}}}
\newcommand{\CX}{{\mathcal {X}}}
\newcommand{\CY}{{\mathcal {Y}}}
\newcommand{\CZ}{{\mathcal {Z}}}

\newcommand{\ab}{{\mathrm{ab}}}
\newcommand{\Ad}{{\mathrm{Ad}}}
\newcommand{\an}{{\mathrm{an}}}
\newcommand{\Aut}{{\mathrm{Aut}}}

\newcommand{\Br}{{\mathrm{Br}}}
\newcommand{\bs}{\backslash}
\newcommand{\bbs}{\|\cdot\|}

\newcommand{\Ch}{{\mathrm{Ch}}}
\newcommand{\cod}{{\mathrm{cod}}}
\newcommand{\cont}{{\mathrm{cont}}}
\newcommand{\cl}{{\mathrm{cl}}}
\newcommand{\criso}{{\mathrm{criso}}}
\newcommand{\de}{{\mathrm{d}}}
\newcommand{\dR}{{\mathrm{dR}}}
\newcommand{\df}{\mathrm{det}^*}
\newcommand{\disc}{{\mathrm{disc}}}
\newcommand{\Div}{{\mathrm{Div}}}
\renewcommand{\div}{{\mathrm{div}}}

\newcommand{\Eis}{{\mathrm{Eis}}}
\newcommand{\End}{{\mathrm{End}}}

\newcommand{\Frob}{{\mathrm{Frob}}}

\newcommand{\Gal}{{\mathrm{Gal}}}
\newcommand{\GL}{{\mathrm{GL}}}
\newcommand{\GO}{{\mathrm{GO}}}
\newcommand{\GSO}{{\mathrm{GSO}}}
\newcommand{\GSp}{{\mathrm{GSp}}}
\newcommand{\GSpin}{{\mathrm{GSpin}}}
\newcommand{\GU}{{\mathrm{GU}}}
\newcommand{\BGU}{{\mathbb{GU}}}

\newcommand{\Hom}{{\mathrm{Hom}}}
\newcommand{\Hol}{{\mathrm{Hol}}}
\newcommand{\HC}{{\mathrm{HC}}}
\newcommand{\id}{\mathrm{id}}
\newcommand{\Img}{{\mathrm{Im}}}
\newcommand{\Ind}{{\mathrm{Ind}}}
\newcommand{\inv}{{\mathrm{inv}}}
\newcommand{\Isom}{{\mathrm{Isom}}}

\newcommand{\Jac}{{\mathrm{Jac}}}
\newcommand{\JL}{{\mathrm{JL}}}

\newcommand{\Ker}{{\mathrm{Ker}}}
\newcommand{\KS}{{\mathrm{KS}}}

\newcommand{\Lie}{{\mathrm{Lie}}}

\newcommand{\new}{{\mathrm{new}}}
\newcommand{\NS}{{\mathrm{NS}}}

\newcommand{\ord}{{\mathrm{ord}}}
\newcommand{\ol}{\overline}
\newcommand{\otf}{\otimes^*}
\newcommand{\rank}{{\mathrm{rank}}}

\newcommand{\PGL}{{\mathrm{PGL}}}
\newcommand{\PSL}{{\mathrm{PSL}}}
\newcommand{\Pic}{\mathrm{Pic}}
\newcommand{\Prep}{\mathrm{Prep}}
\newcommand{\Proj}{\mathrm{Proj}}

\newcommand{\Picc}{\mathcal{P}ic}

\renewcommand{\Re}{{\mathrm{Re}}}
\newcommand{\Res}{{\mathrm{Res}}}
\newcommand{\red}{{\mathrm{red}}}
\newcommand{\reg}{{\mathrm{reg}}}
\newcommand{\sm}{{\mathrm{sm}}}
\newcommand{\sing}{{\mathrm{sing}}}
\newcommand{\SL}{\mathrm{SL}}
\newcommand{\Sp}{\mathrm{Sp}}
\newcommand{\Sym}{{\mathrm{Sym}}}

\newcommand{\tor}{{\mathrm{tor}}}
\newcommand{\tr}{{\mathrm{tr}}}

\newcommand{\ur}{{\mathrm{ur}}}

\newcommand{\vol}{{\mathrm{vol}}}

\newcommand{\ds}{\displaystyle}

\tableofcontents

\section{Introduction}
The goal of this article is to explicitly compute the Kodaira--Spencer map over Hilbert--Siegel modular varieties and twisted Hilbert modular varieties. Our result contains two parts: one is the computation 
of images over integral models, the other is the comparison of metrics in complex setting.
The definition of such map is canonical and easy to 
understand, but it turns out to be difficult to compute the explicit constants since the
choices of identifications are not canonical. 
This paper can be viewed as a generalization of \cite{Yuan2}, which gives the explicit
formula of the Kodaira--Spencer map over Shimura curves. 

Let $F$ be a totally real number field with degree $g$ over $\QQ$. Let
$B=F$ in the first case, and 
$B$ a totally indefinite quaternion algebra over $F$ in the second case.
We always use $d_B$ to denote the discriminant of this quaternion algebra
in the second case.
In the first case, let $(V,\psi)$ be a $2r$-dimensional
symplectic space over $F$, while in the second case
let $V=B$ be a symplectic $B$-module with symplectic form $\psi$ defined by the
reduced trace. Note that we can define Shimura data $(G,X)$ for both cases. 

After choosing some open compact group $U=\prod_p U_p\subset G(\mathbb{A}_f)$, in either case, we can then 
define the corresponding PEL Shimura varieties, denoted by $\XU$, which is over $\CC$. More precisely, we have
\begin{equation*}
    X_U:=\GSp(V)\backslash\HH_r^{g,\pm}\times G(\BA_f)/U
\end{equation*}
in the first case and
\begin{equation*}
    X_U:=B^\times\backslash \mathcal{H}^{g,\pm}\times B_{\mathbb{A}_f}^\times/U
\end{equation*}
in the second case.
We call the Shimura varieties in the first case Hilbert--Siegel modular varieties,
and the second case twisted Hilbert modular varieties. Furthermore, we fix
a maximal order $\mathcal{O}_B\subset B$ in both cases.
We also let $n=n(U)$ be the product of primes $p$ such that $U_p\ne\mathcal{O}^\times_{B,p}$. 
We define $\XXU$ over $\ZZn$ to be the canonical integral model, which is a stack
such that for any $\ZZn$-scheme $S$, 
$\XXU(S)$ parameterizes families of abelian schemes over $S$ with PEL-conditions.
We refer to $\S$\ref{sec Sh} for details.

Denote by $\pi:\mathcal{A}\longrightarrow\XXU$ the universal abelian scheme 
and $\epsilon:\mathcal{X}_U\longrightarrow\mathcal{A}$ the zero section in both cases.
Then we can define relative differential sheaves 
$\Omega_{\mathcal{A}/\mathbb{Z}[\frac{1}{n}]}$, $\Omega_{\mathcal{A}/\mathcal{X}_U}$,
and $\Omega_{\mathcal{X}_U/\mathbb{Z}[\frac{1}{n}]}$, the relative dualizing sheaves
$\omega_{\mathcal{A}/\mathcal{X}_U}$ and $\omega_{\mathcal{X}_U/\mathbb{Z}[\frac{1}{n}]}$ , 
and the relative tangent sheaves $T_{\mathcal{X}_U/\mathbb{Z}[\frac{1}{n}]}$ and
$T_{\mathcal{A}/\mathcal{X}_U}$ . Furthermore, we have the Lie algebra
\begin{equation*}
    \Lie(\mathcal{A}):=\epsilon^*T_{\mathcal{A}/\mathcal{X}_U}\cong \pi_*T_{\mathcal{A}/\mathcal{X}_U},
\end{equation*}
and the Hodge bundles
\begin{equation*}
    \underline{\Omega}_{\mathcal{A}}:=\epsilon^*\Omega_{\mathcal{A}/\mathcal{X}_U}\cong \pi_*\Omega_{\mathcal{A}/\mathcal{X}_U},\quad
    \underline{\omega}_{\mathcal{A}}:=\epsilon^*\omega_{\mathcal{A}/\mathcal{X}_U}\cong \pi_*\omega_{\mathcal{A}/\mathcal{X}_U}.
\end{equation*}

Endow the Hodge bundle $\underline{\omega}_{\mathcal{A}}$ on $\mathcal{X}_U$ with the Faltings metric $\lVert\cdot\rVert_{\Fal}$ and endow $\omega_{\mathcal{X}_U/\mathbb{Z}[\frac{1}{n}]}$ 
on $\mathcal{X}_U$ with the Petersson metric $\lVert\cdot\rVert_{\Pet}$.  We define the Faltings metric
\begin{equation*}
    \lVert\alpha\rVert_{\Fal}^2:=\frac{1}{(2\pi)^l}\left|\int_{\mathcal{A}_x(\mathbb{C})}\alpha\wedge\bar{\alpha}\right|
\end{equation*}
for any section $\alpha\in\underline{\omega}_{\mathcal{A}}(x)\cong\Gamma(\mathcal{A}_x,\omega_{\mathcal{A}_x/\mathbb{C}})$, where $l\in\ZZ$
depends on the two cases. We also define the Petersson metric
\begin{equation*}
    \lVert\de\tau\rVert_{\Pet}:=2^{\frac{gr(r+1)}{2}}\prod_{i=1}^g\det(Y_i)^{\frac{r+1}{2}}
\end{equation*}
for Hilbert--Siegel modular varieties and
\begin{equation*}
    \lVert\de\tau\rVert_{\Pet}:=2^{g}\prod_{i=1}^g\Img(\tau_i)
\end{equation*}
for twisted Hilbert modular varieties. Here in both cases,
$\de\tau$ is a nowhere-vanishing
section of $\omega_{\mathcal{X}_U/\mathbb{Z}[\frac{1}{n}]}$, while 
$Y_i$ and $\tau_i$ denote some points in upper half space.
See $\S$\ref{defks} for details.

Finally, we introduce the Kodaira--Spencer map. Consider the exact sequence
\begin{equation*}
    0\longrightarrow\pi^*\Omega_{\mathcal{X}_U/\mathbb{Z}[\frac{1}{n}]}\longrightarrow\Omega_{\mathcal{A}/\mathbb{Z}[\frac{1}{n}]}\longrightarrow\Omega_{\mathcal{A}/\mathcal{X}_U}\longrightarrow 0.
\end{equation*}
Applying derived functors of $\pi_*$ gives a connecting morphism
\begin{equation*}
    \phi_0:\pi_*\Omega_{\mathcal{A}/\mathcal{X}_U}\longrightarrow R^1\pi_*(\pi^*\Omega_{\mathcal{X}_U/\mathbb{Z}[\frac{1}{n}]}).
\end{equation*}
We call this the Kodaira--Spencer map.

The key point is that we can construct a morphism between two line bundles
$\underline{\omega}_{\mathcal{A}}$ and $\omega_{\mathcal{X}_U/\mathbb{Z}[\frac{1}{n}]}$
from the Kodaira--Spencer map. This allows us to consider the beautiful relation between
these two line bundles.

Now we present the main results of this paper in our two cases.
\begin{theorem}
In the case of Hilbert--Siegel modular varieties, there is a canonical isomorphism
\begin{equation*}
    \psi:\underline{\omega}_\mathcal{A}^{\otimes r+1}\longrightarrow\omega_{\mathcal{X}_U/\mathbb{Z}[\frac{1}{n}]}.
\end{equation*}    
Moreover, under $\psi$, we have $\lVert\cdot\rVert_{\Fal}^{r+1}=\lVert\cdot\rVert_{\Pet}$.
\end{theorem}

\begin{theorem}\label{m1}
In the case of twisted Hilbert modular varieties, there is a canonical morphism
\begin{equation*}
    \psi:\underline{\omega}_\mathcal{A}^{\otimes 2}\longrightarrow\omega_{\mathcal{X}_U/\mathbb{Z}[\frac{1}{n}]}^{\otimes 2},
\end{equation*}    
and its image is the subsheaf $d_B^g\WU^{\otimes 2}$.
Moreover, under $\psi$, we have $\lVert\cdot\rVert_{\Fal}^{2}=\lVert\cdot\rVert_{\Pet}^2$.
\end{theorem}

Note that the second result is a generalization of the main theorem in \cite{Yuan2}.

Our explicit expression of the Kodaira--Spencer map is useful in many situations. For example,
in \cite{YZ}, which studies the average Colmez conjecture, the authors use an explicit
expression of the Kodaira--Spencer map. Moreover, keeping the above notations, we can compute
arithmetic intersection numbers of the above hermitian line bundles. 
We illustrate this idea in the case of twisted Hilbert modular varieties. 
As a consequence of our Theorem \ref{m1},
\begin{equation*}
    \frac{\widehat{\deg}(\hat{c}_1(\omega_{\mathcal{X}_{U}/\ZZ},\lVert\cdot\rVert_\Pet)^{g+1})}{\deg(\omega_{\mathcal{X}_{U,\QQ}/\QQ}^g)}=
    \frac{\widehat{\deg}(\hat{c}_1(\underline{\omega}_{\mathcal{A}},\lVert\cdot\rVert_{\Fal})^{g+1})}{\deg(\underline{\omega}_{\mathcal{A},\QQ}^g)}+g\log d_B.
\end{equation*}
Here $\widehat{\deg}{(\hat{c}_1(\omega,\lVert\cdot\rVert)^{g+1})}$ means the 
arithmetic self-intersection number of any hermitian line bundle 
$(\omega,\lVert\cdot\rVert)$ over
$\XXU$, and $\deg(\omega_\QQ^g)$ means the self-intersection number of the 
underlying line bundle $\omega_\QQ$ over generic fiber. Note that this result 
is closely related to \cite{BKG}
which studies the intersection theory over Hilbert modular surfaces.
This formula also allows us to find compatibility betweeen
the main theorem in \cite{Yuan3} and the formula of \cite{KRY} over $\QQ$.

The organization of this paper is given as follows.
In $\S$\ref{sec Sh}, we recall the definition of Hilbert--Siegel
modular varieties and twisted Hilbert modular varieties. 

In $\S$\ref{sec KS}, we introduce the notion of
the Kodaira--Spencer map, which establishes a canonical map between the Hodge bundle induced 
from the universal abelian scheme and the relative differential sheaf over 
the Shimura variety. We will also consider their determinants,
which are line bundles, and define some suitable metrics on these line bundles. Next, we 
construct a morphism between these metrized line bundles from the Kodaira--Spencer map,
which we also call the Kodaira--Spencer map.
Our main theorems are Theorem \ref{main1} and Theorem \ref{main2},
which establish an explicit relation between two metrized line bundles. We also check
the image over integral models in this seciton.

In $\S$\ref{sec metric}, we introduce a useful trick to compute an isomorphism between the
tangent space of a complex abelian variety and its first structure sheaf cohomology 
explicitly, which is in the
language of \v{C}ech cohomology. Then we give an explicit formula of the Kodaira--Spencer map 
in the complex setting, which also implies the explicit formula of our morphism between line 
bundles. Finally, we complete our proof of the main theorem by comparing two metrics on each 
line bundles using the explicit expression.

\textbf{Acknowledgement}
The author would like to thank professor Xinyi Yuan for a lot of help, without whom he
 would never complete this paper. He would like to thank his friend Weixiao Lu for many helpful advice.
 He would also like to thank Roy Zhao and professor Liang Xiao for many
 useful communications.

\section{Hilbert--Siegel modular varieties and twisted Hilbert modular varieties} \label{sec Sh}
In this section, we review some basics on PEL Shimura varieties. We will give a clear 
definition of Hilbert--Siegel modular varieties and twisted Hilbert modular varieties, which are
the PEL Shimura varieties we work with in this paper.

\subsection{Siegel upper half space}
The story of Shimura varieties always starts at hermitian symmetric domains, i.e., 
symmetric hermitian manifolds of noncompact type. Among these 
spaces, there is a useful example which is crucial to Siegel modular varieties, the Siegel 
upper half space.

Recall the Siegel upper half space $\mathcal{H}_r$ of degree $r$ consists of all the symmetric
complex matrices $Z=X+iY$, where $X$ and $Y$ are real $r\times r$ matrices with $Y$ positive definite.
The map $Z=(Z_{ij})\longmapsto (Z_{ij})_{j\le i}$ identifies $\mathcal{H}_r$ with an open 
subset of $\mathbb{C}^{r(r+1)/2}$.

Note that there is a natural transitive action of $\Sp(2r,\mathbb{R})$ on $\mathcal{H}_r$ by
\begin{equation*}
    \begin{pmatrix} A & B \\ C & D \end{pmatrix}Z=(AZ+B)(CZ+D)^{-1},
\end{equation*}
where $A,B,C,D$ are $r\times r$ real matrices. Clearly, $\begin{pmatrix} 0 & I_r \\ -I_r & 0 \end{pmatrix}$ acts as an involution on $\mathcal{H}_r$.
We remark that in the case of $r=1$, this is the classical upper half plane, with 
$\Sp(2,\mathbb{R})\cong\SL(2,\mathbb{R}) $.

One remarkable fact is that $\mathcal{H}_r$ has modular interpretation similar to upper-half plane,
i.e., $\mathcal{H}_r$ can be identified with the set of $J\in\Hom(V)$ such that $J^2=-id$ and
$\psi(v,Jv)>0$ for any $v\in V$. Here $(V,\psi)$ is a symplectic space of dimension $2r$ over 
$\mathbb{R}$. More
precisely, $Z=X+iY$ corresponds to the matrix $\begin{pmatrix} X & Y \\ W & -X \end{pmatrix}$ where
$X^2+WY=-I_r$, while this matrix is in one-to-one corresponds with $J$ after choosing a symplectic 
basis.
Furthermore, $Z=X+iY$ also gives a lattice $\mathbb{Z}^r\oplus Z\cdot\mathbb{Z}^r$ in 
$\mathbb{C}^r$, denote by $\Lambda_Z$, and then defines an abelian variety $\mathbb{C}^r/\Lambda_Z$.
This complex torus is indeed an abelian variety due to the existence of a Riemann form.

In this paper, we will work on the relative differential sheaf $\Omega_{\mathcal{H}_r/\mathbb{C}}$ and the relative dualizing sheaf $\omega_{\mathcal{H}_r/\mathbb{C}}$.
We denote by $\{\de Z_{ij}\}_{j\le i}$ the basis of $\Omega_{\mathcal{H}_r/\mathbb{C}}$ and $\wedge^{j\le i}\de Z_{ij}$
the basis of $\omega_{\mathcal{H}_r/\mathbb{C}}$.
Meanwhile, there is only one metric on $\mathcal{H}_r$ up to scaling whose isometry group is $\Sp(2r,\mathbb{R})$, but we 
do not use this fact in this paper. Instead, we are interested in some metric on the line bundle
$\omega_{\mathcal{H}_r/\mathbb{C}}$.

\subsection{Hilbert--Siegel modular variety}
Now we introduce the concept of a Hilbert--Siegel modular variety. This definition is a generalization of
the definition of a Siegel modular variety, and we sometimes use the name 
Hilbert modular variety for simplicity throughout this paper.

We give the definition of Hilbert--Siegel modular varieties.
In this case, we choose $F$ to be a totally 
real field of degree $g$ over $\mathbb{Q}$, $B=F$ is the trivial central algebra over $F$,
while $V$ is a dimension $2r$ symplectic space over $F$ with symplectic form $\psi'$.
Take $\psi=\tr_{F/\QQ}\circ\psi'$, we complete the definition. Note that when $F=\QQ$, we will
obtain Siegel modular varieties in this way. In fact, it follows that 
$G'=\Res_{F/\QQ}\Sp(V,\psi')$ and $G=\Res^\circ_{F/\QQ}\GSp(V,\psi')$.
Here $\Res^\circ_{F/\QQ}$ means for any algebraic group $H/F$ with similitude map
$v:H\longrightarrow\mathbb{G}_m$, define 
\begin{equation*}
    \Res^\circ_{F/\QQ}H:=(\Res_{F/\QQ}v)^{-1}\mathbb{G}_m.
\end{equation*}
We also conclude that the $G(\RR)$-conjugacy class $X=\HH_r^{g,\pm}$ with 
each connected component $X^+\cong\HH_r^g$. Here $\HH_r^{g,\pm}=(\HH_r^+)^g\sqcup(\HH_r^-)^g$.
Furthermore, if we set $\{\sigma_i\}_{i=1,...,g}$ to be distinct embeddings of $F$ into $\RR$,
each $\sigma_i$ then induces an action of $\GSp(V)_\RR$ on $\HH_r^\pm$. Hence, taking all 
$\sigma_i$, there is an action of $\Res^\circ_{F/\QQ}\GSp(V)$ on $\HH_r^{g,\pm}$, which is the one defining
Shimura varieties.

We further set $\mathcal{O}_F\subset F$ the ring of
integers, which is also the maximal order. We also fix a lattice $\Lambda\subset V$, and
for convenience, we choose a suitable scalar such that $\psi$ is unimodular when restricted
to $\Lambda$, which is always possible.

Denote by $\mathbb{A}_f$ the ring of finite ad\`{e}les. 
Let $U=\prod_p U_p\subset G(\mathbb{A}_f)$ be an open compact subgroup. One then defines
\begin{equation*}
    n:=n(U)=\prod_{p:U_p\ne G(\mathbb{Z})\otimes_{\ZZ}\ZZ_p}p
\end{equation*}
to be the product of primes $p$ such that $U_p$ is not maximal. Here 
$G(\ZZ)$ consists of $g\in G(\QQ)$ with 
$g\Lambda(\mathcal{O}_F)=\Lambda(\mathcal{O}_F)$. 
By assumption $U$ contains $U(N)$ for some positive integer $N$,
which is  the principal open subgroup of $G(\mathbb{A}_f)$:
\begin{equation*}
    U(N):=\{g\in G(\mathbb{A}_f)| g\Lambda(\hat{\mathcal{O}}_F)=\Lambda(\hat{\mathcal{O}}_F),\quad
    (g-\id)\Lambda(\hat{\mathcal{O}}_F)\subset N\Lambda(\hat{\mathcal{O}}_F)\}.
\end{equation*}
Then we have the Hilbert--Siegel modular variety
\begin{equation*}
    X_U:=\GSp(V)\backslash\HH_r^{g,\pm}\times G(\BA_f)/U.
\end{equation*}

$X_U$ has a canonical integral model $\mathcal{X}_U$ over
$\ZZn$ which is a stack. For any $\ZZn$-schemes $S$, $\XXU(S)$ then represents the modular 
problem of quadruples $(A,\lambda,i,\bar{\eta})$ as follows:\\
(1) $A$ is an abelian scheme over $S$ of relative dimension $rg$;\\
(2) $\lambda$ is a principal polarization of $A$;\\
(3) $i:\mathcal{O}_F\longrightarrow\End_S(A)$ is a ring homomorphism which sends $\id_{\mathcal{O}_F}$ to the
Rosati involution induced by $\lambda$, and the determinant condition
$\det_{\Lie(A)}(x_1b_1+\cdots+x_{rg}b_{rg})=\det_{V^{-1,0}}(x_1b_1+\cdots+x_{rg}b_{rg})$. Here we have the hodge decompostion of $V(\CC)=V^{-1,0}\oplus V^{0,-1}$, where $V^{-1,0}$ is the eigenspace on which
$z\in\CC$ acts as $z$;\\
(4) $\bar{\eta}$ is a $U$-level structure on $A$.

Here we present some details about the $U$-level structure. Note that the 
following discussion also works for twisted Hilbert modular varieties,
which we will define in the next subsection.
The $U$-level structure is defined as below. Choose $N$ dividing a high power of $n$
such that $U$ contains $U(N)$. Consider pairs $(T,\eta)$, where $T\longrightarrow S$ is
an \'{e}tale cover of $S$, and $\eta:\Lambda(\mathcal{O}_B/N\mathcal{O}_B)\longrightarrow A(T)[N]$
is a $\mathbb{Z}$-linear map satisfying $\eta_2=\eta_1\circ u$ for some $u\in U$ acting on 
$\Lambda(\mathcal{O}_B/N\mathcal{O}_B)$, where $\eta_j:\Lambda(\mathcal{O}_B/N\mathcal{O}_B)\longrightarrow A(T\times_S T)[N]$ is the composition
of $\eta$ with the map $p_j^*: A(T)[N]\longrightarrow A(T\times_S T)[N]$ induced by
projection $p_j:T\times_S T\longrightarrow T$ for $j=1,2$. Here $A(T)$ means the
$T$-point of $A$.
Two pairs $(T,\eta)$ and
$(T',\eta')$ are said to be $U$-equivalent if there is an \'{e}tale cover $T"\longrightarrow S$
refining both $T\longrightarrow S$ and $T'\longrightarrow S$ such that $\eta'_{T"}=\eta_{T"}\circ u$
for some $u\in U$, where $\eta_{T"}$ and $\eta'_{T"}$ are maps naturally induced by $\eta$ and
$\eta'$ respectively. Finally, a $U$-level structure on $A$ is an equivalence class of pairs
$(T,\eta)$.

Note that $\mathcal{X}_U$ is a Deligne--Mumford
stack which is flat and a relative local complete intersection over $\ZZn$, and a scheme with the same properties if $U$ contains $U(N)$ for $N>3$, see \cite[Theorem 2.2.2]{Pa}. Even as a Deligne--Mumford stack, $\mathcal{X}_U$
has an \'{e}tale cover by schemes, i.e., for any open subgroup $U'\subset U$ sufficiently small,
the morphism $\mathcal{X}_{U'}\longrightarrow\mathcal{X}_U$ is \'{e}tale over
$\ZZ[\frac{1}{n(U')}]$ with $\mathcal{X}_{U'}$ 
a scheme, so it is sufficient to consider most of our arguments and
terminologies for the case of scheme via \'{e}tale descent.
Furthermore, denote by $\mathcal{X}_U^{\sm}$ the smooth locus of $\mathcal{X}_U$. Due to the fact that $\mathcal{X}_U$ is flat and a relative local complete intersection over $\ZZn$, $\mathcal{X}_U$ is normal and Cohen-Macaulay, and
the non-smooth locus $\mathcal{X}_U^{\sing}:=\mathcal{X}_U\backslash\mathcal{X}_U^{\sm}$ has codimension 2. This will be useful for our later discussion about locally free sheaves.

\subsection{Twisted Hilbert modular variety}
Now we introduce another kind of PEL Shimura varieties, the twisted Hilbert modular variety. 

In this case, let $F$ be a totally 
real field of degree $g$ over $\mathbb{Q}$, $B/F$
be an indefinite quaternion algebra, i.e., for any archimedean place $v$, $B\otimes F_v$ is 
split, and $V=B$ as a left $B$-module of rank 1. 
There is a canonical involution $*$ on $B$.
Denote by $\mathfrak{d}_B$ the discriminant of $B$ as an ideal in $\mathcal{O}_F$, which is the product
of all primes $\mathfrak{p}\subset F$ that are ramified in $B$.  In general, $\mathfrak{d}_B$ is not a principal ideal,
but in our discussion, we always assume that $\mathfrak{d}_B$ is generated by one element $d_B$ in $\mathcal{O}_F$. 
This will simplify our discussion. We remark that when $\mathfrak{d}_B$ is not a principal ideal,
the polarization part in PEL modular interpretation is not a principal polarization, making
things more difficult, but our main result still holds.
For convenience, in our paper we always use $d_B$ for both the $F$-ideal and
the corresponding element in $\mathcal{O}_F$.
We refer to \cite{Voi} for useful properties of quaternions algebras.
Throughout this paper, we always fix a set of isomorphisms
\begin{equation*}
    \sigma_i: B\otimes_F \mathbb{R}\cong M_2(\mathbb{R}),\quad i=1,...,g,
\end{equation*}
one for each archimedean place of $F$.
Note that we can give $B$ a natural symplectic structure using the reduced trace $\tr:B\longrightarrow F$. More precisely, fix $a\in B$ which is a pure quaternion, i.e., $a^2\in F$
but $a\notin F$, then we choose $(V,\psi)=(B,\tr_{F/\mathbb{Q}}\circ \tr(auv^*))$.
It is routine to check that this construction indeed provides us with a $(B,*)$-symplectic module structure on $V$. 

With these settings, by the definition of $(G,X)$, we conclude immediately that $G(\mathbb{Q})\subset B^\times$ consists of those $b\in B^\times$ with
$q(b)\in\QQ$,
while $G'(\mathbb{Q})=B_1$, where $B_1=\{b\in B|q(B)=1\}$ for the reduced norm $q:B\longrightarrow F$. We also use $B^\times_+=\{b\in B|q(b)>0\}$, where $q(b)>0$ means for any
archimedean place $v$, $q(b)_v>0$. We call this kind of $b$ totally positive.
Furthermore, in this case we have $X^+=\mathcal{H}^g$, since
$\sigma:B\otimes_\mathbb{Q}\mathbb{R}\cong M_2(\mathbb{R})^g$.

Note that the reflex field of a twisted Hilbert modular variety is always $\mathbb{Q}$. 
Now in order to have a canonical integral model, let $\mathcal{O}_F$ be the ring of integer in $F$,
we fix a maximal order $\mathcal{O}_B$ of $B$ throughout the paper. Since in our case $V=B$, $\mathcal{O}_B$ itself becomes a lattice in $B$. But we need to modify such a lattice since $\psi$
may not be unimodular when restricting to $\mathcal{O}_B$. Denote by $\mathfrak{D}_F$, 
$\mathfrak{D}_F^{-1}$ and $d_F$ the different, codifferent and discriminant of $\mathcal{O}_F/\mathbb{Z}$ respectively. If $\tr(auv^*)$ has image $\mathfrak{l}\subset F$, where
$\mathfrak{l}$ is an $F$-fractional ideal, then we set $\Lambda:=\mathfrak{l}^{-1}\mathfrak{D}_F^{-1}\otimes_{\mathcal{O}_F}\mathcal{O}_B$ to be the
desired lattice. It is then obvious to check that $\psi$ is unimodular when restricting to $\Lambda$.
Hence we can still denote $\hat{\mathcal{O}}_B=\mathcal{O}_B\otimes_\mathbb{Z}\hat{\mathbb{Z}}$ , while $\mathcal{O}_{B,p}=\mathcal{O}_B\otimes_\mathbb{Z}\mathbb{Z}_p$ for any prime $p$. We further have
$\mathcal{O}_{B,\mathfrak{p}}=\mathcal{O}_B\otimes_{\mathcal{O}_F}\mathcal{O}_{F,\mathfrak{p}}$
for any prime $\mathfrak{p}\subset F$ over $p$. Then by classical algebraic number theory we have
$\mathcal{O}_{B,p}=\mathcal{O}_B\otimes\prod_{\mathfrak{p}|p}\mathcal{O}_{F,\mathfrak{p}}$.

Let $U=\prod_p U_p\subset \hat{\mathcal{O}}_B^\times$ be an open compact subgroup, then define
\begin{equation*}
    n:=n(U)=\prod_{p:U_p\ne \mathcal{O}_{B,p}^\times}p
\end{equation*}
to be the product of primes $p$ such that $U_p$ is not maximal.
By assumption $U$ contains $U(N)$ for some positive integer $N$, where
\begin{equation*}
    U(N):=\prod_{p\nmid N}\mathcal{O}_{B,p}^\times\times\prod_{p|N}(1+N\mathcal{O}_{B,p})^\times.
\end{equation*}
Note this definition agrees with the previous one.

Thus we conclude that in this case, the twisted Hilbert modular variety is
\begin{equation*}
    X_U:=B^\times\backslash \mathcal{H}^{g,\pm}\times B_{\mathbb{A}_f}^\times/U=B^\times_+\backslash \mathcal{H}^g \times B_{\mathbb{A}_f}^\times/U.
\end{equation*}
Here $B^\times$ acts on $\mathcal{H}^{g,\pm}$ via $\sigma=(\sigma_1,...,\sigma_g)$.

As our discussion about general cases above, $X_U$ has a canonical integral model $\mathcal{X}_U$ over $\mathbb{Z}[\frac{1}{n}]$ which is a stack for $n=n(U)$ as above.  For any
$\mathbb{Z}[\frac{1}{n}]$-scheme $S$, $\mathcal{X}_U(S)$ represents the modular problem
of triples $(A,i,\bar{\eta})$ as follows:\\
(1) $A$ is an abelian scheme over $S$ with pure relative dimension $2g$;\\
(3) $i:\mathcal{O}_B\longrightarrow\End_S(A)$ is a ring homomorphism with the determinant condition
$\det_{\Lie(A/S)}(i(b))=\det(q(b))$;\\
(4) $\bar{\eta}$ is a $U$-level structure on $A$.

Note that $\det(q(b))$ is defined as the image of $q(b)$ under composition of canonical isomorphism
$F\otimes_\mathbb{Q}\mathbb{R}\cong\mathbb{R}^g$ and product map $\mathbb{R}^g\longrightarrow\mathbb{R}$. 
An important point is that although we did not state any $O_F$-linear
polarization in the moduli problem, polarization turns out to be automatic. In fact, fix an element
$\mu\in\mathcal{O}_B$ throughout this paper
such that $\mu^2\in\mathcal{O}_F$ with $(\mu^2)=d_B$ as an ideal, and we require $\mu^2$ to be a totally negative element, i.e., under each embedding to $\mathbb{R}$ by $\sigma_i$, the image of  $\mu^2$ is always negative. This gives a positive involution
\begin{equation*}
    B\longrightarrow B,\quad \bar{\beta}:=\mu^{-1}\beta^*\mu,\quad \beta\in B.
\end{equation*}
Clearly this involution stabilizes $\mathcal{O}_B$, and there exists a unique principal polarization
$\lambda$ whose Rosati involution on $\End_S(A)$ is compatible with this positive involution via $i$.
We refer to \cite{Yuan2} for details.

One remark is that although in this paper, we only deal with the case that $V=B$, one can
follow the method in this paper to obtain the same kind of result in the case that 
$V=B\otimes V'$ for some symplectic space $V'/F$. This makes the name twisted Hilbert
modular variety correspond to the name Hilbert modular variety more natural.

One more thing we remark is that, since there is a natural embedding $\mathcal{O}_F\hookrightarrow\mathcal{O}_B$, (3) implies an $\mathcal{O}_F$ action on $A$. This will be
useful in our later discussion.

\section{Kodaira--Spencer map and its image over the integral model}\label{sec KS}
In this section, we review the definition of the Kodaira--Spencer map, and construct a morphism between two line bundles after some modification. Then we give the statement of our main theorems, and prove the first part
of our main theorems. 

\subsection{Kodaira--Spencer map}\label{defks}
Following all the notations about PEL shimura variety above, denote
the universal abelian scheme by $\pi:\mathcal{A}\longrightarrow\mathcal{X}_U$ with 
$\lambda:\mathcal{A}\longrightarrow\mathcal{A}^t$ the universal principal polarization. Also
denote by $\epsilon:\mathcal{X}_U\longrightarrow\mathcal{A}$ and $\epsilon^t:\mathcal{X}_U\longrightarrow\mathcal{A}^t$ the identity sections.

Let $\Omega_{\mathcal{A}/\mathbb{Z}[\frac{1}{n}]}$, $\Omega_{\mathcal{A}/\mathcal{X}_U}$
and $\Omega_{\mathcal{X}_U/\mathbb{Z}[\frac{1}{n}]}$ be the relative differential sheaves,
$\omega_{\mathcal{A}/\mathcal{X}_U}$ and $\omega_{\mathcal{X}_U/\mathbb{Z}[\frac{1}{n}]}$ be
the relative dualizing sheaves, $T_{\mathcal{A}/\mathcal{X}_U}$ and $T_{\mathcal{X}^\sm_U/\mathbb{Z}[\frac{1}{n}]}$ be the relative tangent sheaves. Furthermore,
we have the Lie algebra
\begin{equation*}
    \Lie(\mathcal{A}):=\epsilon^*T_{\mathcal{A}/\mathcal{X}_U}\cong \pi_*T_{\mathcal{A}/\mathcal{X}_U},
\end{equation*}
and the Hodge bundles
\begin{equation*}
    \underline{\Omega}_{\mathcal{A}}:=\epsilon^*\Omega_{\mathcal{A}/\mathcal{X}_U}\cong \pi_*\Omega_{\mathcal{A}/\mathcal{X}_U},
    \underline{\omega}_{\mathcal{A}}:=\epsilon^*\omega_{\mathcal{A}/\mathcal{X}_U}\cong \pi_*\omega_{\mathcal{A}/\mathcal{X}_U}.
\end{equation*}
These definitions clearly imply canonical isomorphisms
\begin{equation*}
    \underline{\Omega}_{\mathcal{A}}\cong \Lie(\mathcal{A})^{\vee},
    \underline{\omega}_{\mathcal{A}}\cong \det\underline{\Omega}_{\mathcal{A}}.
\end{equation*}

In fact, the Hodge bundle $\underline{\omega}_{\mathcal{A}}$ on $\mathcal{X}_U$ has an important
metric called the Faltings metric $\lVert\cdot\rVert_{\Fal}$ as follows. 
Define $l=gr$ and $l=2g$ for the Hilbert modular varieties and the twisted Hilbert modular varieties
respectively. For any point
$x\in\mathcal{X}_U(\mathbb{C})$, let $\alpha\in\underline{\omega}_{\mathcal{A}}(x)\cong\Gamma(\mathcal{A}_x,\omega_{\mathcal{A}_x/\mathbb{C}})$, where $\mathcal{A}_x$ is the fiber of $\mathcal{A}$ above $x$ which is a dimension $l$ complex abelian variety, then $\alpha$ is a holomorphic $l$-form. The Faltings metric is defined by
\begin{equation*}
    \lVert\alpha\rVert_{\Fal}^2:=\frac{1}{(2\pi)^l}\left|\int_{\mathcal{A}_x(\mathbb{C})}\alpha\wedge\bar{\alpha}\right|.
\end{equation*}
See \cite{Yuan1} about this canonical hermitian metric for example. 

Meanwhile, as we remarked before, $\mathcal{X}_U$ is a local complete intersection, with non-smooth locus 
codimension 2.
Hence the dualizing sheaf $\omega_{\mathcal{X}_U/\mathbb{Z}[\frac{1}{n}]}$ is a line bundle on 
$\mathcal{X}_U$, which is canonically isomorphic to 
$\det(\Omega_{\mathcal{X}_U/\mathbb{Z}[\frac{1}{n}]})$. In general, we can
endow $\omega_{\mathcal{X}_U/\mathbb{Z}[\frac{1}{n}]}$ on $\mathcal{X}_U$ with the Petersson
metric $\lVert\cdot\rVert_{\Pet}$, and we define this metric in two cases respectively.

In the case of twisted Hilbert modular variety, $\dim\XU=gr(r+1)/2$ and we can choose 
$(Z_1,...,Z_g)$ to be the coordinates of $\HH_r^g$, where each $Z_i$ is a symmetric complex matrix
with positive definite imaginary part $Y_i$. Now denote by 
$\de\tau_i:=\wedge^{j\ge k} \de Z_{i,jk}$, where $Z_{i,jk}$ is the $(j,k)$-element in matrix
$Z_i$ for $1\le i\le g$. Then clearly $\de\tau:=\wedge_{i=1}^g\de\tau_i$ is a nowhere-vanishing
section of $\omega_{\mathcal{X}_U/\mathbb{Z}[\frac{1}{n}]}$, and the Petersson metric is 
defined by 
\begin{equation*}
    \lVert\de\tau\rVert_{\Pet}:=2^{\frac{gr(r+1)}{2}}\prod_{i=1}^g\det(Y_i)^{\frac{r+1}{2}}.
\end{equation*}
Note that in the case of $g=1$, this defines the Petersson metric of dualizing sheaf over the 
Siegel modular variety. 

In the case of the twisted Hilbert modular variety, $\dim\XU=g$
and we can denote by $(\tau_1,...,\tau_g)$ the usual coordinate of $\mathcal{H}^g\subset\mathbb{C}^g$, Then the Petersson metric is defined by
\begin{equation*}
    \lVert\de\tau\rVert_{\Pet}:=2^{g}\prod_{i=1}^g\Img(\tau_i).
\end{equation*}
Here $\Img$ means to take the imaginary part.
Note that for $g=1$, this definition is the same as the Petersson metric defined on dualizing sheaf
of Shimura curve, see \cite{Yuan2} for example.

Now we are ready to introduce the Kodaira--Spencer map as follows. 
Note that the following discussion holds in general case. Consider the exact sequence
\begin{equation*}
    0\longrightarrow\pi^*\Omega_{\mathcal{X}_U/\mathbb{Z}[\frac{1}{n}]}\longrightarrow\Omega_{\mathcal{A}/\mathbb{Z}[\frac{1}{n}]}\longrightarrow\Omega_{\mathcal{A}/\mathcal{X}_U}\longrightarrow 0.
\end{equation*}
Apply derived functors of $\pi_*$, it gives a connecting morphism
\begin{equation*}
    \phi_0:\pi_*\Omega_{\mathcal{A}/\mathcal{X}_U}\longrightarrow R^1\pi_*(\pi^*\Omega_{\mathcal{X}_U/\mathbb{Z}[\frac{1}{n}]}).
\end{equation*}
We call this the Kodaira--Spencer map. Note that there are canonical isomorphisms
\begin{equation*}
    R^1\pi_*(\pi^*\Omega_{\mathcal{X}_U/\mathbb{Z}[\frac{1}{n}]})\longrightarrow R^1\pi_*\mathcal{O}_{\mathcal{A}}\otimes\Omega_{\mathcal{X}_U/\mathbb{Z}[\frac{1}{n}]}\longrightarrow\Lie(\mathcal{A}^t)\otimes\Omega_{\mathcal{X}_U/\mathbb{Z}[\frac{1}{n}]}\longrightarrow\underline{\Omega}_{\mathcal{A}^t}^\vee\otimes\Omega_{\mathcal{X}_U/\mathbb{Z}[\frac{1}{n}]},
\end{equation*}
where $\mathcal{A}^t$ denotes the dual universal abelian scheme. Here I give some explanation.
The first isomorphism holds by projection formula, the second isomorphism is just
$R^1\pi_*\mathcal{O}_{\mathcal{A}}\cong\pi_*^t T_{\mathcal{A}^t/\mathcal{X}_U}$, which follows
from deformation theory, see \cite[page 37]{Mil} for example.
The third one is just our definition. Then the Kodaira--Spencer map can be also written as
\begin{equation}\label{phi1}
    \phi_1:\underline{\Omega}_{\mathcal{A}}\longrightarrow\underline{\Omega}_{\mathcal{A}^t}^\vee\otimes\Omega_{\mathcal{X}_U/\mathbb{Z}[\frac{1}{n}]}.
\end{equation}

\subsection{Isomorphism from deformation theory}\label{construct}
In this subsection we introduce an isomorphism from deformation theory. Taking dualization of $\phi_1$ one has
\begin{equation}
    \phi_2:\underline{\Omega}_{\mathcal{A}}\otimes\underline{\Omega}_{\mathcal{A}^t}\longrightarrow\Omega_{\mathcal{X}_U/\mathbb{Z}[\frac{1}{n}]}.
\end{equation}
Recall that $\mathcal{X}_U^{\sm}$ means the smooth locus of $\mathcal{X}_U$ as above, with non-smooth
locus of codimension 2 in a suitable sense. By deformation theory, we have a crucial theorem
which is valid for general PEL Shimura varieties. 

\begin{theorem} \label{deformation}
There is an isomorphism
\begin{equation}
    \phi_3:(\underline{\Omega}_{\mathcal{A}}\otimes\underline{\Omega}_{\mathcal{A}})|_{\mathcal{X}_U^{\sm}}/\mathcal{R}\longrightarrow\Omega_{\mathcal{X}_U^{\sm}/\mathbb{Z}[\frac{1}{n}]}.
\end{equation}
Here $\mathcal{R}$ is the subsheaf of $(\underline{\Omega}_{\mathcal{A}}\otimes\underline{\Omega}_{\mathcal{A}})|_{\mathcal{X}_U^{\sm}}$
locally generated by
\begin{equation*}
    (i(\beta)^*u)\otimes v-u\otimes(i(\beta)^*v),
\end{equation*}
\begin{equation*}
    (\lambda^*(u^t)\otimes v-\lambda^*(v^t)\otimes u),\quad 
    u,v\in\underline{\Omega}_{\mathcal{A}},\beta\in\mathcal{O}_B
\end{equation*}
where $i$ is the endomorphism and $\lambda$ is the principal polarization.
\end{theorem}

\begin{proof}
We refer to \cite[Proposition 2.3.5.2]{Lan} for this beautiful fact.
Note that we only need the result over smooth locus.
From there we have an isomorphism induced from $\phi_2$
\begin{equation}
    \phi_4:(\underline{\Omega}_{\mathcal{A}}\otimes\underline{\Omega}_{\mathcal{A}^t})|_{\mathcal{X}_U^{\sm}}/\mathcal{R}\longrightarrow\Omega_{\mathcal{X}_U^{\sm}/\mathbb{Z}[\frac{1}{n}]}.
\end{equation}
Here $\mathcal{R}$ is the subsheaf of $(\underline{\Omega}_{\mathcal{A}}\otimes\underline{\Omega}_{\mathcal{A}^t})|_{\mathcal{X}_U^{\sm}}$
locally generated by
\begin{equation*}
    (i(\beta)^*u)\otimes v-u\otimes((i(\beta)^t)^*v),
\end{equation*}
\begin{equation*}
    (\lambda^*(u^t)\otimes v-\lambda^*(v)\otimes u^t),\quad u\in\underline{\Omega}_\mathcal{A},
    v\in\underline{\Omega}_{\mathcal{A}^t},\beta\in\mathcal{O}_B.
\end{equation*}
Then our theorem is a straight forward corollary after taking the principal polarization.
\end{proof}

\subsection{Morphism between line bundles}
In this subsection we will derive the morphism in our main theorems.
In order to obtain a morphism between line bundles $\underline{\omega}_\mathcal{A}$ and
$\omega_{\mathcal{X}_U/\mathbb{Z}[\frac{1}{n}]}$, one observation is that since we have embedding
$\mathcal{O}_F\hookrightarrow\mathcal{O}_B$, there is also an $\mathcal{O}_F$-action.
This also implies that there is a natural $\mathcal{O}_F\otimes_\ZZ\mathcal{O}_{\mathcal{X}^\sm}$-module structure of
$\OA$, $\WA$, $\OU$ and $\WU$. In the later discussion, we always use the
notation $\otf$ for the tensor product of two coherent sheaves
over $\mathcal{O}_F\otimes_\ZZ\mathcal{O}_{\mathcal{X}^\sm}$. For simplicity,
we still use $\otimes$ for $\otimes_{\mathcal{O}_{\mathcal{X}^\sm}}$, and
$\det$ for determinant over $\mathcal{O}_{\mathcal{X}^\sm}$. Also note that
since $\XXU$ is a local complete intersection, morphisms between line bundles 
over $\mathcal{X}_U^\sm$ can
also be extended to morphisms over $\XXU$.

Now we define an operator $\df$ for any coherent sheaf over
$\mathcal{O}_F\otimes_\ZZ\mathcal{O}_{\mathcal{X}^\sm}$. 
First, for any locally free sheaf $\mathcal{F}$ over $\mathcal{O}_F\otimes_\ZZ\mathcal{O}_{\mathcal{X}^\sm}$, 
 $\df(\mathcal{F})$
is defined by taking top exterior power over $\mathcal{O}_F\otimes_\ZZ\mathcal{O}_{\mathcal{X}^\sm}$, which is still a locally free sheaf. In general, for any coherent sheaf $\mathcal{F}$ over $\mathcal{O}_F\otimes_\ZZ\mathcal{O}_{\mathcal{X}^\sm}$, we can always find a
resolution 
\begin{equation*}
    0\longrightarrow\mathcal{G}_m\longrightarrow ...\longrightarrow\mathcal{G}_0\longrightarrow\mathcal{F}\longrightarrow 0,
\end{equation*}
where $\mathcal{G}_i$ are all locally free sheaves over
$\mathcal{O}_F\otimes_\ZZ\mathcal{O}_{\mathcal{X}^\sm}$ for $i=1,...,m$,
and this is a long exact sequence for $\mathcal{O}_F\otimes_\ZZ\mathcal{O}_{\mathcal{X}^\sm}$-coherent sheaves.
Note that $\mathcal{O}_F\otimes_\ZZ\mathcal{O}_{\mathcal{X}^\sm}$ is regular, 
since $\mathcal{O}_F$ is regular and we work on the smooth locus of $\XXU$.
Then, the finiteness of the above resolution is implied by the finiteness of 
projective dimension, which follows from regularity.

 We then define
\begin{equation*}
    \df(\mathcal{F})=\otimes_{i=0}^{*\ m}\df(\mathcal{G}_i)^{(-1)^i}.
\end{equation*}
Here $\df(\mathcal{G})^{-1}$ means the dual of $\df(\mathcal{F})$. Obviously,
this definition is independent of the choice of resolution.

Unlike $\det(\mathcal{F})$, which is a line bundle, $\df(\mathcal{F})$
has rank $g=\deg(F/\QQ)$. In general, $\df$ has
many similar properties as $\det$, and the proofs are elementary. We list some properties,
\begin{lemma}\label{det}
    $\mathcal{F}$ is a coherent sheaf over
$\mathcal{O}_F\otimes_\ZZ\mathcal{O}_{\mathcal{X}^\sm}$, then there is a canonical isomorphism
\begin{equation*}
    \det(\mathcal{F})\cong\det(\df(\mathcal{F})).
\end{equation*}
Moreover, if $\mathcal{F}$ is a rank 1 locally free sheaf over
$\mathcal{O}_F\otimes_\ZZ\mathcal{O}_{\mathcal{X}^\sm}$, for any morphism of sheaves 
$\mathcal{M}\longrightarrow\mathcal{N}\otf\mathcal{F}$, $\mathcal{M}$ and $\mathcal{N}$ 
are locally free sheaves on the same scheme which are rank $r$ over
$\mathcal{O}_F\otimes_\ZZ\mathcal{O}_{\mathcal{X}^\sm}$,
then there is a natural morphism 
\begin{equation*}
    \df(\mathcal{M})\longrightarrow\df(\mathcal{N})\otf
    \mathcal{F}^{\otf r}.
\end{equation*}
Here $\mathcal{F}^{\otf r}$ means $\mathcal{F}$ tensored with itself over
$\mathcal{O}_F\otimes_\ZZ\mathcal{O}_{\mathcal{X}^\sm}$ for $r$-times.
\end{lemma}
Note that it is sufficient to prove the case when $\mathcal{F}$ is locally free
over $\mathcal{O}_F\otimes_\ZZ\mathcal{O}_{\mathcal{X}^\sm}$ by definition.
Thus everything is covered in linear algebra.

In the case of Hilbert modular variety, we claim that
\begin{equation*}
    \det((\underline{\Omega}_{\mathcal{A}}\otimes
    \underline{\Omega}_{\mathcal{A}})/\mathcal{R})\cong
    \WA^{\otimes r+1}.
\end{equation*}
To prove the claim, note that in this case $\mathcal{O}_B=\mathcal{O}_F$, hence taking the quotient of
$\underline{\Omega}_\mathcal{A}\otimes\underline{\Omega}_\mathcal{A}$ by
the subsheaf generated by $(i(\beta)^*u)\otimes v-u\otimes(i(\beta)^*v)$ is isomorphic to 
$\underline{\Omega}_\mathcal{A}\otf\underline{\Omega}_\mathcal{A}$.
It remains to check $\det(\underline{\Omega}_\mathcal{A}\otf\underline{\Omega}_\mathcal{A}/\mathcal{R}')$, where $\mathcal{R}'$ is
 the subsheaf generated by $(\lambda^*(u^t)\otimes v-\lambda^*(v^t)\otimes u)$.
Then we have
\begin{equation*}
    \df(\det(\underline{\Omega}_\mathcal{A}\otf\underline{\Omega}_\mathcal{A}/\mathcal{R}'))\cong(\df\OA)^{\otf r+1}.
\end{equation*}
Indeed, this result holds by linear algebra when $\OA$ is locally free over
$\mathcal{O}_F\otimes_\ZZ\mathcal{O}_{\mathcal{X}^\sm}$, hence by our definition
of $\df$, this holds in general.

Applying Lemma \ref{det}, we then conclude that there is a natural isomorphism
\begin{equation*}
    \WA^{\otimes r+1}\longrightarrow\WU.
\end{equation*}

In the case of twisted Hilbert modular variety, it is even simpler to construct such morphism.
Recall the Kodaira--Spencer map \eqref{phi1}, apply Lemma \ref{det} again, we then have a morphism
\begin{equation*}    
    \df(\OA)\otf\df(\underline{\Omega}_{\mathcal{A}^t})\longrightarrow
    \OU^{\otf 2}.
\end{equation*}
This morphism has both sides rank $g$ locally free sheaves, hence we can 
take the principal polarization on the left side first, then take determinant
of both sides to obtain a morphism between line bundles, i.e., there is a morphism
\begin{equation*}
    \WA^{\otimes 2}\longrightarrow\WU^{\otimes 2}.
\end{equation*}

Now, we are ready to state the main theorems in this paper. In the case of Hilbert modular 
varieties, we have:
\begin{theorem}\label{main1}
There is a canonical isomorphism
\begin{equation}
    \psi:\underline{\omega}_\mathcal{A}^{\otimes r+1}\longrightarrow\omega_{\mathcal{X}_U/\mathbb{Z}[\frac{1}{n}]}.
\end{equation}    
Moreover, under $\psi$, we have $\lVert\cdot\rVert_{\Fal}^{r+1}=\lVert\cdot\rVert_{\Pet}$.
\end{theorem}

While in the case of twisted Hilbert modular varieties, we also have:
\begin{theorem}\label{main2}
There is a canonical morphism
\begin{equation}
    \psi:\underline{\omega}_\mathcal{A}^{\otimes 2}\longrightarrow\omega_{\mathcal{X}_U/\mathbb{Z}[\frac{1}{n}]}^{\otimes 2},
\end{equation}    
and its image is the subsheaf $d_B^g\WU^{\otimes 2}$.
Moreover, under $\psi$, we have $\lVert\cdot\rVert_{\Fal}^{2}=\lVert\cdot\rVert_{\Pet}^2$.
\end{theorem}

We will prove these two theorems later.

One remark is that for Hilbert modular varieties, when $g=1$, this result is a beautiful
identity on Siegel upper half plane, which is useful in modular form.
While for twisted Hilbert modular varieties, when $g=1$, our result agree with \cite{Yuan2}.

\subsection{Image over the integral model}
In this subsection we prove the first part of our main theorems, i.e, we check the image
over integral models. Note that for Theorem \ref{main1}, the isomorphism over integral model 
is immediately from our process in defining such morphism between line bundles.
It remains to check our first statement in Theorem \ref{main2}.
The key idea follows \cite{Yuan2} is to check our statement at each non-archimedean 
places. Since $\XXU^{\sing}$ the singular locus has codimension 2 in $\XXU$, it suffices to
check our statement over $\XXU^{\sm}$.

For convenience, denote
\begin{equation*}
    \mathcal{T}=\Lie(\mathcal{A}),\quad\mathcal{T}'=\Lie(\mathcal{A}^t)^\vee,\quad
    \mathcal{N}=\HHom_{\mathcal{O}_B}(\mathcal{T}',\mathcal{T}).
\end{equation*}
Recall from Theorem \ref{deformation}, if we change tensor product to homomorphism, there is 
then an isomorphism
\begin{equation*}
    \phi_5:T_{\XXU/\ZZn}\longrightarrow\mathcal{N}
\end{equation*}
over the smooth locus. Note that this isomorphism also implies that $\mathcal{N}$ is a 
rank 1 $\mathcal{O}_F\otimes_\ZZ\mathcal{O}_{\mathcal{X}^\sm}$-module.

There is then a canonical composition
\begin{equation*}
    \mathcal{N}\otf\mathcal{T}'\longrightarrow\mathcal{T}.
\end{equation*}
Now we take $\df$ on both sides and apply Lemma \ref{det}, we have
\begin{equation*}
    \mathcal{N}^{\otf2}\otf\df(\mathcal{T}')
    \longrightarrow\df(\mathcal{T}).
\end{equation*}
Note that on both sides, we have a rank $g$ locally free sheaf over $\XXU$.
Taking determinants on both sides and applying Lemma \ref{det} again, we finally have
\begin{equation*}
    \det(\mathcal{N})^{\otimes2}\otimes\det(\mathcal{T}')\longrightarrow\det(\mathcal{T}).
\end{equation*}
Then our first statement is equivalent to an isomorphism
\begin{equation*}
    \det(\mathcal{N})^{\otimes2}\otimes\det(\mathcal{T}')
    \longrightarrow d_B^g\det(\mathcal{T}),
\end{equation*}
which is also equivalent to an isomorphism
\begin{equation}\label{integral}
    \mathcal{N}^{\otf2}\otf\df(\mathcal{T}')
    \longrightarrow d_B\df(\mathcal{T}).
\end{equation}

In the rest of this section, we are going to check the isomorphism \eqref{integral} locally
at each non-archimedean places of $F$. There are two classes of these places, one are those
$\mathfrak{p}$ coprime to $d_B$, which we call them good primes.
The other are those $\mathfrak{p}|d_B$, which we call them bad primes. For simplicity,
we always use subscript $\mathfrak{p}$ for the base change over $\mathcal{O}_{F,\mathfrak{p}}$.

\begin{lemma}
In the case of $\mathfrak{p}\nmid d_B$, \eqref{integral} induces an isomorphism 
\begin{equation*}    
     \mathcal{N}_{\mathfrak{p}}^{\otf2}\otf
     \df(\mathcal{T}_{\mathfrak{p}}')\longrightarrow \df(\mathcal{T}_\mathfrak{p}).
\end{equation*}
\end{lemma}

\begin{proof}
By the definition of discriminant, $\mathcal{O}_{B,\mathfrak{p}}$ is isomorphic to 
$M_2(\mathcal{O}_{F,\mathfrak{p}})$.
Fix such an isomorphism and take idempotents $e_1=\begin{pmatrix} 1 & 0 \\0 & 0\end{pmatrix}$ and $e_2=\begin{pmatrix} 0 & 0 \\0 & 1\end{pmatrix}$, while the involution is given by
$\epsilon=\begin{pmatrix} 0 & 1\\ 1 & 0\end{pmatrix}$. Then the idempotent decomposition implies
\begin{equation*}
    \mathcal{T}_\mathfrak{p}=e_1\mathcal{T}_\mathfrak{p}\oplus e_2\mathcal{T}_\mathfrak{p},\quad
    \mathcal{T}'_\mathfrak{p}=e_1\mathcal{T}'_\mathfrak{p}\oplus e_2\mathcal{T}'_\mathfrak{p}.
\end{equation*}
Note that the involution $\epsilon$ gives isomorphisms $e_1\mathcal{T}_\mathfrak{p}\longrightarrow e_2\mathcal{T}_\mathfrak{p}$ and $e_1\mathcal{T}'_\mathfrak{p}\longrightarrow e_2\mathcal{T}'_\mathfrak{p}$. These are all
locally free over $\mathcal{X}_{U,\mathfrak{p}}$.

Note that there are canonical isomorphisms for $j=1,2$
\begin{equation*}
    \mathcal{N}_\mathfrak{p}\longrightarrow\HHom_{\mathcal{O}_B}(e_j\mathcal{T}'_\mathfrak{p},e_j\mathcal{T}_\mathfrak{p})\longrightarrow (e_j\mathcal{T}_\mathfrak{p})\otimes(e_j\mathcal{T}'_p)^{\vee}.
\end{equation*}
Thus we conclude that
\begin{equation*}
    \mathcal{N}_\mathfrak{p}^{\otimes 2}\longrightarrow (e_1\mathcal{T}_\mathfrak{p})
    \otimes(e_1\mathcal{T}'_\mathfrak{p})^\vee
    \otimes(e_2\mathcal{T}_\mathfrak{p})
    \otimes(e_2\mathcal{T}'_\mathfrak{p})^\vee.
\end{equation*}
is an isomorphism, which implies our lemma since $\det(\mathcal{T}_\mathfrak{p})\cong(e_1\mathcal{T}_\mathfrak{p})\otimes(e_2\mathcal{T}_\mathfrak{p})$.
\end{proof}

We also have the following lemma for the other case.
\begin{lemma}
In the case of $\mathfrak{p}| d_B$, \eqref{integral} induces an isomorphism 
\begin{equation*}    
     \mathcal{N}_{\mathfrak{p}}^{\otf 2}\otf
     \df(\mathcal{T}_{\mathfrak{p}}')\longrightarrow \mathfrak{p}\df(\mathcal{T}_\mathfrak{p}).
\end{equation*}
\end{lemma}
\begin{proof}
This case is more complicated than the first case. Note that for any quaternion
algebra over a local field, it splits after passing to the unique unramified quadratic extension. 
Thus in this case,
instead of passing to $\mathcal{O}_{F,\mathfrak{p}}$, we will pass to 
$\mathcal{O}_{F,\mathfrak{p}^2}$, the valuation ring of the unique unramified quadratic extension. 
Moreover, taking \'{e}tale covering allows us to compute everything over smooth scheme.

For convenience, we use $D$ for $B\otimes_F F_{\mathfrak{p}}$ and $\mathcal{O}_D$ for its unique maximal order
$\mathcal{O}_B\otimes_{\mathcal{O}_F}\mathcal{O}_{F,\mathfrak{p}}$. By assumption, $D$ is division.
We also denote by $\varpi$ the uniformizer of $\mathcal{O}_{F,\mathfrak{p}}$, and sometimes 
for convenience it also denotes the uniformizer of $\mathcal{O}_{F,\mathfrak{p}^2}$, which 
depends on the situation.

Now, since the problem is local, by taking any closed point in $\mathcal{X}_{\mathfrak{p}^2}$, 
we denote by $R$ the  flat noetherian local ring over $\mathcal{O}_{F,\mathfrak{p}^2}$
such that $\mathfrak{p}R$ is a prime ideal, which is the stalk at the closed point.
We still use $\varpi$ for the generator of $\mathfrak{p}R$.
Denote by $T$ and $N$ the fibers of locally free sheaves 
$\mathcal{T}_{\mathfrak{p}^2}$ and $\mathcal{N}_{\mathfrak{p}^2}$ at the closed point,
here the subscript $\cdot_{\mathfrak{p}^2}$ means passing to $\mathcal{O}_{F,\mathfrak{p}^2}$.
Then, they are $R$-modules with a left $\mathcal{O}_D$-action satisfying determinant condition.
Here, the determinant condition means $\det(i(\beta))=q(\beta)$ for any $\beta\in\mathcal{O}_D$.
Now the problem is reduced to check that 
\begin{equation*}
    (N\otimes_R N)\otimes_R \det(T')\longrightarrow\det(T)
\end{equation*}
given by the determinant of the composition map 
\begin{equation*}
    N\otimes_R T'\longrightarrow \Hom_R(T',T)\longrightarrow T
\end{equation*}
is injective with image $\mathfrak{p}\det(T)$.

To work out this purely algebraic question, first we need to translate every objects into matrices.
Recall that we always have a decomposition for $\mathcal{O}_D$, i.e., there exists some $j\in\mathcal{O}_D$
with
\begin{equation*}
    \mathcal{O}_D=\mathcal{O}_{F,\mathfrak{p}^2}\oplus\mathcal{O}_{F,\mathfrak{p}^2}j
\end{equation*}
such that
\begin{equation*}
    j^*=-j,\quad j^2=\varpi,\quad jx=\bar{x}j(x\in\mathcal{O}_{F,\mathfrak{p}^2}).
\end{equation*}
Here $\bar{x}$ denotes the conjugation of $x$ under Galois action.
I refer to \cite[Theorem 13.3.11]{Voi} for this proposition about quaternion algebra over local field.
Based on this result, there is an $R$-linear ring isomorphism
\begin{equation*}
    \tau:\mathcal{O}_D\otimes_{\mathcal{O}_{F,\mathfrak{p}}}R\longrightarrow
    \begin{pmatrix}
        R & R\\ \varpi R & R
    \end{pmatrix}.
\end{equation*}
Indeed, it is sufficient to check the case $R=\mathcal{O}_{F,\mathfrak{p}^2}$. In this case we 
can take an injective homomorphism
\begin{equation*}
    \mathcal{O}_D\longrightarrow M_2(\mathcal{O}_{F,\mathfrak{p}^2}),\quad x\longrightarrow
    \begin{pmatrix}
        x &  \\ & \bar{x}
    \end{pmatrix},\quad j\longrightarrow
    \begin{pmatrix}
         & 1 \\ \varpi & 
    \end{pmatrix}, x\in \mathcal{O}_{F,\mathfrak{p}^2},
\end{equation*}
while we base change to $\mathcal{O}_{F,\mathfrak{p}^2}$ this becomes
\begin{equation*}
    \mathcal{O}_D\otimes_{\mathcal{O}_{F,\mathfrak{p}}}\mathcal{O}_{F,\mathfrak{p}^2}\longrightarrow M_2(\mathcal{O}_{F,\mathfrak{p}^2}),\quad x\otimes a\longrightarrow
    \begin{pmatrix}
        ax &  \\ & a\bar{x}
    \end{pmatrix},\quad j\longrightarrow
    \begin{pmatrix}
         & 1 \\ \varpi & 
    \end{pmatrix}, x\otimes a\in \mathcal{O}_{F,\mathfrak{p}^2}\otimes\mathcal{O}_{F,\mathfrak{p}^2}.
\end{equation*}
Clearly this implies $\tau$ is an isomorphism.

Next, we claim that up to isomorphism there are exactly two $R$-modules with left $\mathcal{O}_D$-action 
satisfying determinant condition. Under the identification by $\tau$, denote by $e_1=\begin{pmatrix} 1 & 0 \\0 & 0\end{pmatrix}$ and $e_2=\begin{pmatrix} 0 & 0 \\0 & 1\end{pmatrix}$ the idempotent elements in
$\mathcal{O}_D\otimes_{\mathcal{O}_{F,\mathfrak{p}}}R$, and $j=\begin{pmatrix} 0 & 1 \\ \varpi & 0\end{pmatrix}$.
Then for any such module $M$, there is an idempotent decomposition 
\begin{equation*}
    M=e_1 M\oplus e_2 M,\quad e_1 M=Rx,\quad e_2 M=Ry
\end{equation*}
as $R$-modules, where $e_1 M,e_2 M$ are projective hence free over $R$ with rank 1 by the determinant condition,
and $x,y$ are their bases. Note that since $je_1=e_2 j$ and $j e_2=e_1 j$, the action of $j$ on $M$
switches $e_1 M$ and $e_2 M$. Thus, in order to determine such rank 2 module $M$, it is equivalent to determine
the action of $j$ on the bases $x,y$, i.e., assume
\begin{equation*}
    jx=ay,\quad jy=bx,\quad a,b\in R,
\end{equation*}
then $a,b$ determines $M$ up to isomorphism. Note that $ab=\varpi$ since $j^2=\varpi$. But $\varpi$ is the 
uniformizer, hence either $a\in R^\times$ or $b\in R^\times$. For $b\in R^\times$, one can replace $x$
by $b^{-1}x$, so the action becomes $jx=\varpi y,jy=x$. Similarly for the case $a\in R^\times$,
we have $jx=y,jy=\varpi x$. This provides with two $R$-modules with left $\mathcal{O}_D$-action 
satisfying determinant condition, and they are not isomorphic since the map induced by $j$-action
$e_1 M\longrightarrow e_2 M$ is surjective in exactly one case.

Furthermore, it is clear from the construction of these two $R$-modules with left $\mathcal{O}_D$-action 
satisfying the determinant condition that for such module $T$, the dual $T^\vee$ is not isomorphic to $T$.
Indeed, the dualizing process switches two idempotent components, hence permutes $x$ and $y$.

Finally, we are ready to complete our proof. Without loss of generality, we assume 
\begin{equation*}
    T=e_1 T\oplus e_2 T,\quad e_1 T=Rx,\quad e_2 T=Ry,\quad jx=\varpi y,\quad jy=x,
\end{equation*}
and then
\begin{equation*}
    T'=e_1 T'\oplus e_2 T',\quad e_1 T'=Rx',\quad e_2 T'=Ry',\quad jx'=y',\quad jy'=\varpi x'.
\end{equation*}
Clearly, there are canonical maps
\begin{equation*}
    \Hom_{\mathcal{O}_D}(T',T)\longrightarrow\Hom_R(e_j T',e_j T),\quad j=1,2.
\end{equation*}
Moreover, if we consider the $j$-action more carefully, i.e., assume $x'$ maps to $ax$ and $y'$ 
maps to $by$ for some $a,b\in R$, then $j$-action implies $b=\varpi a$. Since $e_1,e_2$ and $j$ 
determines the entire action of $\mathcal{O}_D$, we are able to conclude that 
\begin{equation*}
    \Hom_{\mathcal{O}_D}(T',T)\longrightarrow\Hom_R(e_1 T',e_1 T),\quad\Hom_{\mathcal{O}_D}(T',T)\longrightarrow\varpi\Hom_R(e_2 T',e_2 T)
\end{equation*}
are both isomorphisms. Hence, we have isomorphism
\begin{equation*}
    \Hom_{\mathcal{O}_D}(T',T)\otimes_R T'\longrightarrow e_1 T\oplus\varpi e_2 T.
\end{equation*}
Taking determinants on both sides, we finish the proof of this lemma.
\end{proof}

Combine the above two lemmas, we then check all the non-archimedean places which are relative 
prime to $n$. Thus we finish the proof of the first statement in our main theorem.

\section{Comparison of the metrics}\label{sec metric}
In this section we explicitly compare two metrics in the complex setting. Since everything is now 
over $\mathbb{C}$, we only need to understand the following two exact sequences
\begin{equation*}    
     0\longrightarrow\pi^*\Omega_{\mathcal{H}_r/\mathbb{C}}\longrightarrow\Omega_{\mathcal{A}/
     \mathbb{C}}     
     \longrightarrow\Omega_{\mathcal{A}/\mathcal{H}_r}\longrightarrow 0,
\end{equation*}
which is for Hilbert modular varieties, and
\begin{equation*}    
     0\longrightarrow\pi^*\Omega_{\mathcal{H}^g/\mathbb{C}}\longrightarrow\Omega_{\mathcal{A}/
     \mathbb{C}}     
     \longrightarrow\Omega_{\mathcal{A}/\mathcal{H}^g}\longrightarrow 0,
\end{equation*}
which is for twisted Hilbert modular varieties.
Here $\pi:\mathcal{A}\longrightarrow\mathcal{H}^g$ denotes the universal abelian scheme, which 
is in fact an algebraic stack. Recall our construction of the map in the main theorem
\ref{main1} and \ref{main2}, it turns out that we
need to find an explicit formula for the connecting morphism
\begin{equation}\label{com1}
\phi:\underline{\Omega}_{\mathcal{A}/\mathcal{H}_r}\longrightarrow\Lie(\mathcal{A}/\mathcal{H}_r)\otimes\Omega_{\mathcal{H}_r/\mathbb{C}}
\end{equation}
for Hilbert--Siegel modular varieties, and
\begin{equation}\label{com2}
\phi:\underline{\Omega}_{\mathcal{A}/\mathcal{H}^g}\longrightarrow\Lie(\mathcal{A}/\mathcal{H}^g)\otimes\Omega_{\mathcal{H}^g/\mathbb{C}}
\end{equation}
for twisted Hilbert modular varieties. Note that this $\phi$ 
map corresponds to the morphism \eqref{phi1}.

\subsection{Explicit map between complex tangent space} \label{Cech}
In this section we establish a useful lemma which is important for our later calculation. 

In order to compute $\phi$ explicitly, it is important to understand the following isomorphism
\begin{equation*}
    V\longrightarrow\Lie(A)\longrightarrow\Lie(A^t)\longrightarrow H^1(A,\mathcal{O}_A).
\end{equation*}
Here $V$ is a complex vector space of dimension $g$, $A=V/\Lambda$ is a $g$-dimensional complex abelian variety, with a polarization $\lambda$ induced by a Riemann form $E:\Lambda\times\Lambda\longrightarrow\mathbb{Z}$ for $\Lambda$ a lattice in $V$.
Then the first map comes from $\Lambda$ tensoring $\mathbb{R}$, the second is by the polarization, while
the third map is the canonical isomorphism by deformation of line bundles, which we have mentioned 
in {\S}\ref{sec KS}. 

The main difficulty to give an explicit map of this isomorphism is the map to $H^1(A,\mathcal{O}_A)$.
We will use \v{C}ech cohomology to introduce an explicit canonical map
\begin{equation*}
   h:\Hom_{\mathbb{Z}}(\Lambda,\mathbb{C})\longrightarrow H^1(A,\mathcal{O}_A),
\end{equation*}
which is a special case for the explicit homomorphism
\begin{equation*}
   \delta:H^1(\Lambda,\mathcal{F}(V))\longrightarrow H^1(A,\mathcal{F}_A).
\end{equation*}
Here $\mathcal{F}$ is a sheaf in the complex analytic setting with trivial $\Lambda$-action.

We first need a suitable open cover of $A$ coming from cover of $V$. Take a set $\{U_t\}_{t\in I}$
that is a family of open subsets $U_t\subset V$, satisfying the following condition:\\
(1) each composition $U_t\longrightarrow V\longrightarrow A$ is injective, \\
(2) $\bigcup_{t\in I}U_t\longrightarrow A$ is surjective, \\
(3) for any $t,t'\in I$, the difference of two sets $U_{t'}-U_t$, also as a subset of $V$, contains
at most one point of $\Lambda$. Denote by $c_{t,t'}$ this point if it exists.
Denote by $\bar{U}_t$ the image of $U_t\longrightarrow A$, which then forms a cover of $A$.
Obviously, such a cover exists just by taking each $U_t$ small enough. We call such a cover an admissible
cover of $A$.

Now, we define $\delta$ explicitly. For any cross-homomorphism $\alpha:\Lambda\longrightarrow\mathcal{F}(V)$, define a \v{C}ech cocycle $\delta(\alpha)$
such that each component $\delta(\alpha)_{t,t'}\in\mathcal{F}(\bar{U}_{t,t'})$ on $\bar{U}_{t,t'}:=\bar{U}_t\cap\bar{U}_{t'}$ is given by the image of $\alpha(c_{t,t'})$ under
the composition $\mathcal{F}(V)\longrightarrow\mathcal{F}(U_t)\longrightarrow\mathcal{F}(\bar{U}_t)\longrightarrow\mathcal{F}(\bar{U}_{t,t'})$. This is obviously a \v{C}ech cocycle by the uniqueness
of points $c_{t,t'}$. Hence we define $\delta$ explicitly.

Back to the definition of $h$, one just chooses $\mathcal{F}=\mathcal{O}$, then define $h$ to be
the composition
\begin{equation*}
   \Hom_{\mathbb{Z}}(\Lambda,\mathbb{C})\longrightarrow H^1(\Lambda,\mathbb{C})\longrightarrow H^1(\Lambda,\mathcal{O}(V))\longrightarrow H^1(A,\mathcal{O}_A).
\end{equation*}
Here, the first isomorphism follows from $\Lambda$ acting on $\mathbb{C}$ trivially, and the second
map is induced by natural map $\mathbb{C}\longrightarrow\mathcal{O}(V)$.

Finally, we have the following lemma.
\begin{lemma}\label{app}
The composition
\begin{equation*}
    V\longrightarrow\Lie(A)\longrightarrow\Lie(A^t)\longrightarrow H^1(A,\mathcal{O}_A)
\end{equation*}
is given by
\begin{equation*}
    z\mapsto2\pi i h(E(z,\cdot)),
\end{equation*}
where $E(z,\cdot)$ is viewed as an element in $\Hom_{\mathbb{Z}}(\Lambda,\mathbb{C})$.
\end{lemma}
I refer \cite[page 17]{Yuan2} for a complete proof of this lemma. The key point of the proof is
Appel-Humbert theorem which can be found in \cite{Mum}.

\subsection{Explicit map for Hilbert--Siegel modular varieties}
In this subsection we give an explicit expression of \eqref{com1}. In the case of Hilbert--Siegel
modular varities,
the universal abelian scheme  $\pi:\mathcal{A}\longrightarrow\mathcal{H}^g_r$ is given by
\begin{equation*}
   \mathbb{Z}^{2gr}\backslash(\mathcal{H}^g_r\times\mathbb{C}^{gr}).
\end{equation*}
Without loss of generality, throughout this subsection, we assume $g=1$. For $g>1$, there
is no essential difference. See our computation in the next subsection when $g>1$.

Now, for $g=1$, the action of $\ZZ^{2r}$ is given explicitly by
\begin{equation*}
   (a_1,...,a_r,b_1,...,b_r)\cdot(Z,(z_1,...,z_r))=(Z,(z_1,...,z_r)+(a_1,...,a_r)\cdot Z+(b_1,...,b_r)).
\end{equation*}
Here $a_i,b_i\in\mathbb{Z}$, $Z\in\mathcal{H}_r$ is a matrix, while $z_i\in \mathbb{C}$.
Later, we will use $\theta=(\alpha,
\beta)\in\mathbb{Z}^{2r}$ for $(a_1,...,a_r,b_1,...,b_r)$,
and $\zeta\in\mathbb{C}^r$ for $(z_1,...,z_r)$.
Clearly from this definition, for each $Z\in\mathcal{H}_r$, we have a canonical uniformization
\begin{equation*}
   A_Z:=\mathbb{C}^r/\Lambda_Z,\quad\Lambda_Z:=\mathbb{Z}^r\cdot Z\oplus\mathbb{Z}^r,
\end{equation*}
where we regard elements in $\mathbb{Z}^r$ as row vectors, with $\mathbb{Z}^r\subset\mathbb{R}^r$ and $\mathbb{R}^r\otimes\mathbb{C}\cong\mathbb{C}^r$ in natural
way. Then $\theta$ defined above also means the coordinate of a point in lattice.
Note that there is a canonical positive Riemann form over $\Lambda_Z$
\begin{equation*}
   E:\Lambda_Z\times\Lambda_Z\longrightarrow\mathbb{Z},\quad E(\alpha\cdot Z+\beta,\alpha'\cdot Z+\beta')=-\alpha\cdot\beta'+\alpha'\cdot\beta.
\end{equation*}
Here $\alpha,\alpha',\beta,\beta'\in\mathbb{Z}^r$. 

Moreover, we can choose $\{\de z_i\}_{i=1,...,r}$ as a basis of locally free sheaves $\Omega_{\mathcal{A}/\mathcal{H}_r}$ and $\underline{\Omega}_{\mathcal{A}/\mathcal{H}_r}$. 
Then $\wedge_{i=1}^r\de z_i$ becomes a basis of $\underline{\omega}_{\mathcal{A}/\mathcal{H}_r}$,
and $\{\frac{\partial}{\partial z_i}\}_{i=1,...r}$ becomes a basis of $\Lie(\mathcal{A}/\mathcal{H}_r)$.
Also recall that we used $\{\de Z_{ij}\}_{j\le i}$ as a basis of $\Omega_{\mathcal{H}_r/\mathbb{C}}$.

Here is the main result in this subsection.
\begin{theorem}\label{explicit1}
The Kodaira--Spencer map over Siegel upper half space
\begin{equation*}
    \phi:\underline{\Omega}_{\mathcal{A}/\mathcal{H}_r}\longrightarrow\Lie(\mathcal{A}/\mathcal
    {H}_r)\otimes\Omega_{\mathcal{H}_r/\mathbb{C}}
\end{equation*}
gives
\begin{equation*}
    \de z_j\mapsto \sum_{i=1}^r\frac{1}{2\pi i}(\frac{\partial}{\partial z_i})\otimes\de Z_{ij}.
\end{equation*}
Therefore, the map
\begin{equation*}
    \psi:\underline{\omega}_\mathcal{A}^{\otimes r+1}\longrightarrow\omega_{\mathcal{H}_r/\mathbb{C}}
\end{equation*}
induced by $\phi$ gives
\begin{equation*}
    (\wedge^i \de z_i)^{\otimes r+1}\mapsto(\frac{1}{2\pi i})^{\frac{r(r+1)}{2}}(\de\tau),
\end{equation*}
where $\de\tau$ means $\wedge^{i\ge j} \de Z_{ij}$, which has been mentioned before.
\end{theorem}

\begin{proof}
To begin with, we prove the first statement about $\phi$. Fix $Z^0\in\mathcal{H}_r$, I claim 
that on the fiber above $Z^0$, the connecting map
\begin{equation*}
    \phi_0:\underline{\Omega}_{\mathcal{A}/\mathcal{H}_r}\longrightarrow R^1\pi_*\mathcal{O}_\mathcal{A}\otimes\Omega_{\mathcal{H}_r/\mathbb{C}}
\end{equation*}
is given by
\begin{equation*}
    \de z_j|_{Z^0}\mapsto \sum_{i=1}^r h(\delta_i)\otimes\de Z_{ij}|_{Z^0},
\end{equation*}
where
\begin{equation*}
    h:\Hom_\mathbb{Z}(\Lambda_{Z^0},\mathbb{C})\longrightarrow H^1(\mathcal{A}_{Z^0},\mathcal{O}_{\mathcal{A}_{Z^0}})
\end{equation*}
is the map defined in \ref{Cech} and $\delta_i\in\Hom_\mathbb{Z}(\Lambda_{Z^0},\mathbb{C})$ 
sends $\theta^0=(\alpha^0,\beta^0)\in\Lambda_{Z^0}$ (as a coordinate) to $\alpha_i^0$, i.e., 
the $i$-th term of $\alpha^0\in\mathbb{Z}^r$.

To prove this claim, we again use \v{C}ech cohomology as before. Similar to the case of a single complex abelian variety, since $\mathcal{A}$ has a universal cover $\mathcal{H}_r\times\mathbb{C}^r$, we can take a family of open subsets 
$\{U_t\}_{t\in I}$ satisfying the same kind of admissible condition as above. 
This provides us with an admissible cover $\{\bar{U}_t\}_{t\in I}$ of $\mathcal{A}$.
Note that conditions (1) and (2) are obvious, while condition (3) becomes \\
(3) for any $t,t'\in I$, the difference $U_t-U_{t'}:=\{(Z,\zeta'-\zeta):(Z,\zeta)\in U_t,(Z,\zeta')\in U_{t'}\}$, where $\zeta,\zeta'\in\mathbb{C}^r$, intersects at most one connected 
component of $\Lambda:=\bigcup_Z\Lambda_Z$. If this component does exist, denote it by
$\theta_{t,t'}=(\alpha_{t,t'},\beta_{t,t'})\in\mathbb{Z}^{2r}$. 

Here I explain this condition (3) for better understanding. When $Z\in\mathcal{H}_r$ varies, 
the lattice $\Lambda_Z=\mathbb{Z}^r\cdot Z\oplus\mathbb{Z}^r$ above $Z$ also varies, and each 
point in the lattice becomes a hypersurface as $Z$ takes all points in $\mathcal{H}_r$. This 
makes our universal lattice $\Lambda$ consist of $\mathbb{Z}^{2r}$-connecting components, 
hence we are able to use such $\theta$ to denote them.

Now we compute the image of $\de z_j$. Recall our exact sequence
\begin{equation*}   
    0\longrightarrow\pi^*\Omega_{\mathcal{H}_r/\mathbb{C}}\longrightarrow\Omega_{\mathcal{A}/
    \mathbb{C}}\longrightarrow\Omega_{\mathcal{A}/\mathcal{H}_r}\longrightarrow 0.
\end{equation*}
Here is a standard method of diagram chasing to compute the connecting map.
Note first that each $\de z_j\in\Omega_{\mathcal{A}/\mathbb{C}}(U_t)$ lifts the corresponding section, which is also denoted by $\de z_j$ as a section of $\pi_*\Omega_{\mathcal{A}/\mathcal{H}_r}$.
Denote by $(\de z_j)_t$ the pushforward of $\de z_j$ via $U_t\longrightarrow\bar{U}_t$, then it is
sufficient to compute $(\de z_j)_{t'}-(\de z_j)_t$ on the overlap $\bar{U}_{t,t'}$, which by the
admissible condition is bijective to $U_{t,t'}$.

Recall the definition of $\mathcal{A}$, the isomorphism $U_{t,t'}\cong U_{t',t}$ sends $(Z,\zeta)$ to
$(Z,\zeta+\alpha_{t,t'}\cdot Z+\beta_{t,t'})$. Thus the pull-back of $z_j\in\mathcal{O}(U_{t',t})$
to $\mathcal{O}(U_{t,t'})$ becomes
\begin{equation*}
    z'_j=z_j+\alpha_{t,t'}\cdot Z_{*,j}+\beta_{t,t',j},
\end{equation*}
where $Z_{*,j}$ means the $j$-th column of $Z$, while $\beta_{t,t',j}$ is the $j$-th term in $\beta_{t,t'}$. One concludes that the pull-back of $(\de z_j)_{t'}-(\de z_j)_t$ is 
\begin{equation*}
    \de z'_j-\de z_j=\alpha_{t,t'}\cdot \de Z_{*,j}
\end{equation*}
Thus this implies our claim above immediately after replacing the inner product of vectors to
$r$-terms sum.

Now, to complete our proof about the first statement of this theorem, it remains to apply the
Lemma \ref{app} above to translate $h(\delta_i)$ into an element in $\Lie(\mathcal{A}_{Z^0})$.
In other words, we need to find some $w_i\in\mathbb{C}^r\cong\Lie(\mathcal{A}_{Z^0})$, 
such that $h(\delta_i)$=$(2\pi i)h(E(w_i,\cdot))$, as we have
already chosen the basis of $\Lie(\mathcal{A}/\mathcal{H}_r)$ above.
But this is straight-forward from our choice of the Riemann form $E$, since the relation becomes
\begin{equation*}
    \delta_i(\theta^0)=\alpha_i^0=(2\pi i)E(w_i,\theta^0)
\end{equation*}
which holds for any $\theta^0=(\alpha^0,\beta^0)$. Then $w_i=(0,...,(2\pi i)^{-1},...,0)$, the only nonzero term is $\beta_i^0$. 
This clearly implies our first result.

To prove the second result, note that if we take dualization of $\phi$ and apply the first result, we even have a simpler formula for morphism
\begin{equation*}
    \phi':\underline{\Omega}_{\mathcal{A}/\mathcal{H}_r}^{\otimes 2}\longrightarrow\Omega_{\mathcal{H}_r/\mathbb{C}}.
\end{equation*}
It is obvious that this map is given by
\begin{equation*}
    \de z_i\otimes\de z_j\mapsto\frac{1}{2\pi i}\de Z_{ij}
\end{equation*}
Then recall our discussion in \ref{construct}, we just need to check the induced isomorphism 
\begin{equation*}
    \underline{\Omega}_{\mathcal{A}/\mathcal{H}_r}^{\otimes 2}/\mathcal{R}'\longrightarrow\Omega_{\mathcal{H}_r/\mathbb{C}},
\end{equation*}
where $\mathcal{R}'$ is generated by $\{\de z_i\otimes\de z_j-\de z_j\otimes\de z_i\}_{i,j}$,
see Theorem \ref{deformation}. Finally, we arrive at the isomorphism of line bundles
\begin{equation*}
    \psi:\underline{\omega}_{\mathcal{A}/\mathcal{H}_r}^{\otimes 2}\longrightarrow\omega_{\mathcal{H}_r/\mathbb{C}}
\end{equation*}
gives
\begin{equation*}
    (\wedge^i \de z_i)^{\otimes r+1}\mapsto(\frac{1}{2\pi i})^{\frac{r(r+1)}{2}}(\de\tau),
\end{equation*}
This finishes the proof.
\end{proof}

\subsection{Explicit map for twisted Hilbert modular vaieties}
In this subsection we give an explicit expression of \ref{com2}. The outlines of the statement
and proof in this subsection 
are quite similar to those in the case of Hilbert modular varieties, but as we promised above,
we are going to show the computation when $g>1$.

Note that in this case the universal abelian
scheme $\pi:\mathcal{A}\longrightarrow\mathcal{H}^g$ is given by
\begin{equation*}
    \mathcal{A}=\mathcal{O}_B\backslash(\mathcal{H}^g\times\mathbb{C}^{2g}),
\end{equation*}
where the action of $\mathcal{O}_B$ is given by
\begin{equation*}
    \beta(\tau,z)=(\tau,z+\sigma(\beta)(\tau_1,1,\tau_2,1,...,\tau_g,1)^t).
\end{equation*}
Here $\tau=(\tau_1,...,\tau_g)$ is the row vector represents an element in $\mathcal{H}^g$, while 
$z=(z_1,...,z_{2g})^t$ is the column vector represents an element in $\mathbb{C}^{2g}$, $\sigma(\beta)$
denotes the $2g\times 2g$ matrix by identifications $\sigma_i: B\otimes_F \mathbb{R}\cong M_2(\mathbb{R})$, i.e.,
\begin{equation*}
    \sigma(\beta)=
    \begin{pmatrix}
        \sigma_1(\beta) & & \\ & ... & \\ & & \sigma_g(\beta)
    \end{pmatrix}
\end{equation*}
is a block diagonal matrix. For simplicity, we will use $\beta(\tau,1)^t$ for $\sigma(\beta)(\tau_1,1,\tau_2,1,...,\tau_g,1)^t$. Be careful about this since in our notation $(\tau,1)$ is not
$(\tau_1,...\tau_g,1,...,1)$ but $(\tau_1,1,\tau_2,1,...,\tau_g,1)$!

Note that for each $\tau\in\mathcal{H}^g$, there is an canonical uniformization
\begin{equation*}
    \mathcal{A}_\tau=\mathbb{C}^{2g}/\Lambda_\tau,\quad \Lambda_\tau=\mathcal{O}_B(\tau,1)^t.
\end{equation*}
It is routine to check that $\Lambda_\tau$ is a full lattice in $\mathbb{C}^{2g}$. Moreover, $\mathcal{A}_\tau$
has a canonical positive Riemann form as follow. Recall we fix some $\mu\in\mathcal{O}_B$ such that $(\mu^2)=d_B$
as an ideal in $\mathcal{O}_F$ with $\mu^2\in \mathcal{O}_F$ totally negative. Then, we have a symplectic pairing
\begin{equation*}
    E:\Lambda_\tau\times\Lambda_\tau\longrightarrow\mathbb{Z},\quad E(\beta(\tau,1)^t,\beta'(\tau,1)^t)=\tr_{F/\mathbb{Q}}(\tr(\mu^{-1}\beta\beta'^*)),
\end{equation*}
where the inner $\tr$ means the reduced trace of quaternion algebra. By \cite[Lemma 43.6.16]
{Voi}, we know that in the case of $g=1$, $E$ or $-E$ is positive definite and then a Riemann 
form, which we assume to be $E$ without of losing generality. In our general 
case, the proof is almost the same, so I just sketch the idea how to modify the proof there to obtain a proof
of the general case. To check the alternating part, there is no
difference since taking $\tr_{F/\mathbb{Q}}$ does not change anything. To check that 
$(x,y)\mapsto E_\mathbb{R}(i\beta(\tau,1)^t,\beta'(\tau,1)^t)$ is symmetric and positive definite,
being symmetric is again the same as the case of $F=\mathbb{Q}$, while in the argument of being positive definite,
we replace each $>0$ by being totally positive, i.e., the image under each archimedean place $\sigma_i$ is
positive. Thus, we conclude that $E$ gives principal polarizations
\begin{equation*}
    \lambda:\mathcal{A}\longrightarrow\mathcal{A}^t,\quad \lambda: \mathcal{A}_\tau\longrightarrow\mathcal{A}_\tau^t.
\end{equation*}

Moreover, denote by $\Lambda$ the inverse image in $\mathcal{H}^g\times\mathbb{C}^{2g}$ of the identity section
of $\mathcal{A}\longrightarrow\mathcal{H}^g$, i.e., $\epsilon:\mathcal{H}^g\longrightarrow\mathcal{A}$. Then 
over each $\tau\in\mathcal{H}^g$, the fiber $\Lambda_\tau$ is a lattice in $\mathbb{C}^{2g}$, hence we have
\begin{equation*}
    \Lambda=\bigcup_{\tau\in\mathcal{H}^g}\Lambda_\tau=\bigcup_{\tau\in\mathcal{H}^g}\mathcal{O}_B(\tau,1)^t
    =\{\beta(\tau,1)^t:\beta\in\mathcal{O}_B,\tau\in\mathcal{H}^{g}\}.
\end{equation*}
This provides us with a biholomorphic map
\begin{equation*}
    \Lambda\longrightarrow\mathcal{H}^g\times\mathcal{O}_B,\quad \beta(\tau,1)^t\longrightarrow(\tau,\beta),
\end{equation*}
and then a natural bijection $\pi_0(\Lambda)\longrightarrow\mathcal{O}_B$.

Now we are ready to compute the explicit formula of the Kodaira--Spencer map \eqref{com2}. Note 
that under the notation above, we have
\begin{equation*}
    \underline{\Omega}_{\mathcal{A}/\mathcal{H}^g}=\mathcal{O}_{\mathcal{H}^g}\de z,\quad
    \underline{\omega}_{\mathcal{A}/\mathcal{H}^g}=\mathcal{O}_{\mathcal{H}^g}\bigwedge_{i=1}^{2g}\de z_i,\quad
    \Omega_{\mathcal{H}^g/\mathbb{C}}=\mathcal{O}_{\mathcal{H}^g}\de \tau,\quad
    \omega_{\mathcal{H}^g/\mathbb{C}}=\mathcal{O}_{\mathcal{H}^g}\bigwedge_{i=1}^g\de \tau_i.
\end{equation*}
We also write 
\begin{equation*}
    \mu=\begin{pmatrix}
        \sigma_1(\mu) & & \\ &...& \\ & & \sigma_g(\mu)
    \end{pmatrix}, \quad
    \sigma_i(\mu)=\begin{pmatrix}
        a_i & b_i\\ c_i & d_i
    \end{pmatrix},
\end{equation*}
where $a_i,b_i,c_i,d_i\in \mathbb{R}$ for $i=1,..,g$ are coefficients of these matrices. Now 
we state the main result of this subsection. In fact, our result is a higher dimensional 
analogue of \cite[Theorem 4.2]{Yuan2}.

\begin{theorem}\label{computation2}
The Kodaira--Spencer map
\begin{equation*}  \phi:\underline{\Omega}_{\mathcal{A}/\mathcal{H}^g}\longrightarrow\Lie(\mathcal{A}/\mathcal{H}^g)\otimes\Omega_{\mathcal{H}^g/\mathbb{C}}
\end{equation*}
gives
\begin{equation*}
    \phi(\de z_{2i-1})=\frac{1}{2\pi i}(b_i\frac{\partial}{\partial z_{2i-1}}+d_i\frac{\partial}{\partial z_{2i}})\otimes \de \tau_i,
\end{equation*}
\begin{equation*}
    \phi(\de z_{2i})=-\frac{1}{2\pi i}(a_i\frac{\partial}{\partial z_{2i-1}}+c_i\frac{\partial}{\partial z_{2i}})\otimes \de \tau_i,
\end{equation*}
where $i=1,...,g$.
Therefore, the map
\begin{equation*}
    \psi:\underline{\omega}_{\mathcal{A}/\mathcal{H}^g}^{\otimes 2}\longrightarrow\omega_{\mathcal{H}^g/\mathbb{C}}^{\otimes 2}
\end{equation*}
induced by $\phi$ gives
\begin{equation*}
    \psi\left((\bigwedge_{j=1}^{2g}\de z_j)^{\otimes 2}\right)=\left(\frac{d_B}{(2\pi i)^2}\right)^{g}(\bigwedge_{i=1}^g\de \tau_i)^{\otimes 2}.
\end{equation*}
\end{theorem}
\begin{proof}
Once again, our main idea is to check the result locally at each fiber. Throughout this proof, we fix an
element $\tau_0\in\mathcal{H}^g$. Consider the connecting morphism
\begin{equation*}
    \phi_0:\pi_*\Omega_{\mathcal{A}/\mathcal{H}^g}\longrightarrow R^1\pi_*\mathcal{O}_\mathcal{A}\otimes\Omega_{\mathcal{H}^g/\mathbb{C}},
\end{equation*}
we claim that on the fibers above $\tau_0$,
\begin{equation*}
    \phi_0(\de z_{2i-1})|_{\tau_0}=h(l_{2i-1,2i-1})\otimes\de \tau_i|_{\tau_0},
\end{equation*}
\begin{equation*}
    \phi_0(\de z_{2i})|_{\tau_0}=h(l_{2i,2i-1})\otimes\de \tau_i|_{\tau_0}.
\end{equation*}
Here the canonical map
\begin{equation*}
    h:\Hom_{\mathbb{Z}}(\Lambda_{\tau_0},\mathbb{C})\longrightarrow H^1(\mathcal{A}_{\tau_0},\mathcal{O}_{\mathcal{A}_{\tau_0}})
\end{equation*}
is defined in section \ref{Cech}, while 
\begin{equation*}
    l_{j,k}:\Lambda_{\tau_0}=\sigma(\mathcal{O}_B)(\tau_0,1)^t\longrightarrow\sigma(\mathcal{O}_B)\longrightarrow\mathbb{R}
\end{equation*}
is a map to $\mathbb{R}$ taking the $(j,k)$-coefficient of $\sigma(\mathcal{O}_B)\subset M_{2g}(\mathbb{R})$, 
recall that $\sigma(\mathcal{O}_B)$ means those huge block diagonal matrices mentioned before.

To prove this claim, we again use \v{C}ech cohomology as before. Similarly, we again take an 
admissible cover of $\mathcal{A}$ by the universal cover 
$\mathcal{H}^g\times\mathbb{C}^{2g}$. We take a set $\{U_t\}_{t\in I}$ of open subsets 
$U_t\subset\mathcal{H}^g\times\mathbb{C}^{2g}$ satisfying the following conditions:

(1) every compostion $U_t\longrightarrow\mathcal{H}^{g}\times\mathbb{C}^{2g}\longrightarrow\mathcal{A}$ is 
injective,

(2) the induced map $\bigcup_{t\in I}U_t\longrightarrow\mathcal{A}$ is surjective,

(3) for any $t,t'\in I$, the difference $U_{t'}-U_t=\{(\tau,z'-z):(\tau,z)\in U_t,(\tau,z')\in U_{t'}\}$
in $\mathcal{H}^{g}\times\mathbb{C}^{2g}$ intersects at most one connected component of $\Lambda$ defined before.
Denote by $\beta_{t,t'}\in\mathcal{O}_B$ the element representing this connected component if non-empty,
note that here we use the natural bijection $\pi_0(\Lambda)\longrightarrow\mathcal{O}_B$.

Now we do some standard diagram chasing to compute the connecting morphism of the exact sequence
\begin{equation*}
     0\longrightarrow\pi^*\Omega_{\mathcal{H}^g/\mathbb{C}}\longrightarrow\Omega_{\mathcal{A}/\mathbb{C}}
     \longrightarrow\Omega_{\mathcal{A}/\mathcal{H}^g}\longrightarrow 0.
\end{equation*}
Note that the section $\de z_j$ of $\Omega_{\mathcal{A}/\mathbb{C}}$ over each $U_t$ lifts the section
$\de z_j$ of $\pi_*\Omega_{\mathcal{A}/\mathcal{H}^g}$, where $j=1,...,2g$ corresponds to $2i-1$ or $2i$. 
Denote by $(\de z_j)_t$ the pushforward of $\de z_j$ via $U_t\longrightarrow\bar{U}_t$. Then in order to find 
the image of $\de z_j$ under connecting morphism, it is equivalent to compute $(\de z_j)_{t'}-(\de z_j)_t$ on
the overlap $\bar{U}_{t,t'}=\bar{U}_t\cap\bar{U}_{t'}$, which represents a cocycle in the sense of \v{C}ech
cohomology. We also denote by $U_{t,t'}$ or $U_{t',t}$ the subset of $U_t$ or $U_{t'}$ which bijective to
$\bar{U}_{t,t'}$. By definition, the isomorphism $U_{t,t'}\longrightarrow U_{t',t}$ sends $(\tau,z)$ to
$(\tau,z+\beta_{t,t'}(\tau,1)^t)$. Thus, the pull-back of the function $z_j\in\mathcal{O}(U_{t',t})$ to
$\mathcal{O}(U_{t,t'})$ becomes
\begin{equation*}
    z'_j=z_j+l_{j,2i-1}(\beta_{t,t'})\tau_i+l_{j,2i}(\beta_{t,t'}).
\end{equation*}
Here we write $\sigma_i(\beta_{t,t'})=\begin{pmatrix}
    l_{2i-1,2i-1}(\beta_{t,t'}) & l_{2i-1,2i}(\beta_{t,t'}) \\l_{2i,2i-1}(\beta_{t,t'}) & l_{2i,2i}(\beta_{t,t'}) 
\end{pmatrix}$ for the $i$-th block in the big matrix $\sigma(\beta_{t,t'})$, while the right hand sides is
computed by taking matrices product for $\beta_{t,t'}$ and $(\tau,1)^t$.
From this result, we conclude that the pull-back of $(\de z_j)_{t'}-(\de z_j)_t$ to $U_{t,t'}$ is
\begin{equation*}
    \de z'_j-\de z_j=l_{j,2i-1}(\beta_{t,t'})\de \tau_i.
\end{equation*}
As a consequence, we see that the class $\phi_0(\de z_j)$ is represented by \v{C}ech cocycle
\begin{equation*}
    l_{j,2i-1}((\beta_{t,t'}))_{t,t'\in I}\otimes\de \tau_i.
\end{equation*}
By restricting to the fiber $\mathcal{A}_{\tau_0}$ and taking $j=2i-1,2i$ separately, we prove the claim.

In order to conclude the first statement, now we use Lemma \ref{app} to give an explicit formula for $h(l_{j,2i-1})$. 
For convenience we still use index $i,j$ with $j=2i-1$ or $j=2i$. Since the
lemma implies the map $\mathbb{C}^{2g}\cong\Lie(\mathcal{A}_{\tau_0})\longrightarrow H^1(\mathcal{A}_{\tau_0},\mathcal{O}_{\mathcal{A}_{\tau_0}})$ is given by $\rho:z\mapsto 2\pi ih(E(z,\cdot)),$
we want to find some $w_j\in\mathbb{C}^{2g}$ such that
\begin{equation*}
    h(l_{j,2i-1})=(2\pi i)^{-1}\rho(w_j)=h(E(w_j,\cdot)),
\end{equation*}
i.e., we need to find $w_j\in\mathbb{C}^{2g}$ such that 
$E(w_j,\cdot)=l_{j,2i-1}$ in $\Hom_\mathbb{Z}(\Lambda_{\tau_0},\mathbb{C})$. Set $w_j=\beta_j(\tau_0,1)^t$ for
some $\beta_j\in\mathcal{O}_B$, by the definition of our Riemann form, then it is equivalent to find $\beta_j$
such that
\begin{equation*}
    \tr_{F/\mathbb{Q}}(\tr(\mu^{-1}\beta_j\beta'^*))=l_{j,2i-1}(\beta')
\end{equation*}
for any $\beta'\in\mathcal{O}_B$. Note that the left hand side is 
$\sum_{i=1}^g\sigma_i(\tr(\mu^{-1}\beta_j\beta'^*)$, so we are able to determine the block diagonal matrix
$\sigma(\beta_j)$ block-from-block. Then it is obvious that for any $i'\ne i$, the $2\times2$ block 
$\sigma_{i'}(\beta_j)$ is zero, while for $\sigma_i(\beta_j)$ we have
\begin{equation*}
    \sigma_i(\beta_{2i-1})=\begin{pmatrix}
        0 & b_i \\ 0 & d_i
    \end{pmatrix},
\end{equation*}
\begin{equation*}
    \sigma_i(\beta_{2i})=\begin{pmatrix}
        0 & -a_i \\ 0 & -c_i
    \end{pmatrix}
\end{equation*}
Here $\beta'^*$ is $\tr(\beta')\id-\beta'$.
Thus we conclude that 
\begin{equation*}
    w_{2i-1}=(0,...,b_i,d_i,...,0),
\end{equation*}
\begin{equation*}
    w_{2i}=(0,...,-a_i,-c_i,...,0),
\end{equation*}
where the only nontrivial terms are the $(2i-1)$-th and $2i$-th elements in the vector. Then these vectors
corresponds to $b_i\frac{\partial}{\partial z_{2i-1}}+d_i\frac{\partial}{\partial z_{2i}}$ and
$-a_i\frac{\partial}{\partial z_{2i-1}}-c_i\frac{\partial}{\partial z_{2i}}$ under identification.

Finally, since we have $\phi_0=(\rho\otimes 1)\circ\phi$ by definition, while $h(l_{j,2i-1})=(2\pi i)^{-1}\rho(w_j)$
from the above argument, we have
\begin{equation*}
    \phi_0(\de z_{2i-1})|_{\tau_0}=\frac{1}{2\pi i}\rho(b_i\frac{\partial}{\partial z_{2i-1}}+d_i\frac{\partial}{\partial z_{2i}})\otimes\de \tau_i|_{\tau_0},
\end{equation*}
\begin{equation*}
    \phi_0(\de z_{2i})|_{\tau_0}=\frac{1}{2\pi i}\rho(-a_i\frac{\partial}{\partial z_{2i-1}}-c_i\frac{\partial}{\partial z_{2i}})\otimes\de \tau_i|_{\tau_0}.
\end{equation*}
This complete our first statement.

For the second statement, we apply the first result. Then we know for the morphism
\begin{equation*}
    \underline{\Omega}_{\mathcal{A}/\mathcal{H}_g}\longrightarrow
    T_{\mathcal{A}/\mathcal{H}_g}\otimes\Omega_{\mathcal{H}_g/\mathbb{C}}^{\otimes 2},
\end{equation*}
we have explicit formula
\begin{equation*}
    \de z_{2i-1}\wedge\de z_{2i}\mapsto \frac{a_id_i-b_ic_i}{(2\pi i)^2}(\frac{\partial}{\partial z_{2i-1}}\wedge\frac{\partial}{\partial z_{2i}})\otimes(\de \tau_i)^{\otimes 2},
\end{equation*}
here $i=1,...,g$, and $a_id_i-b_ic_i=\det(\sigma_i(\mu))=d_B$ is an element in 
$\mathcal{O}_F$. Thus, we conclude that the map between line bundle is given by
\begin{equation*}
    \bigwedge_{i=1}^g (\de z_{2i-1}\wedge\de z_{2i})\mapsto \left(\frac{d_B}{(2\pi i)^2}\right)^{g}\bigwedge_{i=1}^g (\frac{\partial}{\partial z_{2i-1}}\wedge\frac{\partial}{\partial z_{2i}}) \otimes (\bigwedge_{i=1}^g\de \tau_i)^{\otimes 2},
\end{equation*}
which is obviously equivalent to our result.
\end{proof}

\subsection{Comparison of metrics}
Finally, we are able to prove the second statements about metrics in our main theorems
\ref{main1} and \ref{main2}. Since everything is 
compatible with pull-back, we only need to compare metric in the complex setting. Once again,
we check two cases separately. Recall our definitions of the Faltings metric and the 
Peterson metric in $\S$\ref{defks}.

In the case of Hilbert modular varieties,
we fix some $Z^0=X^0+iY^0\in\mathcal{H}_r$ and check at the stalk. Once again we assume $g=1$
without loss of generality, and the computation when $g>1$ is given in the case of
twisted Hilbert modular varieties below.
By definition, we have
\begin{equation*}
    \lVert\wedge^{i}\de z_i\rVert_{\Fal}^2(Z^0)=\frac{1}{(2\pi)^{r}}|\int_{\mathcal{A}_{Z^0}}\wedge^i(\de z_i\wedge\de \bar{z}_i)|=\frac{1}{\pi^r}\vol(\mathbb{C}^r/\Lambda_{Z^0}).
\end{equation*}
But by our definition of $\Lambda_{Z^0}$, it is almost trivial that 
\begin{equation*}
    \vol(\mathbb{C}^r/\Lambda_{Z^0})=\det(Y^0).
\end{equation*}
Hence we conclude that
\begin{equation*}
    \lVert\wedge^{i}\de z_i\rVert_{\Fal}^2=\frac{1}{\pi^r}\det(Y).
\end{equation*}
Moreover, we have the Petersson metric 
\begin{equation*}
    \lVert\de \tau\rVert_{\Pet}=2^{\frac{r(r+1)}{2}}(\det Y)^{\frac{r+1}{2}}.
\end{equation*}
Then apply Theorem \ref{explicit1}, especially the correspondence
\begin{equation*}
    \Phi\left((\wedge^i \de z_i)^{\otimes r+1}\right)\longrightarrow(\frac{1}{2\pi i})^{\frac{r(r+1)}{2}}(\de\tau),
\end{equation*}
we conclude that
\begin{equation*}
    \lVert\cdot\rVert_{\Fal}^{r+1}=\lVert\cdot\rVert_{\Pet}.
\end{equation*}

Similarly, in the case of twisted Hilbert modular variety,
we check locally for each $\tau_0\in\mathcal{H}^g$. Then we have
\begin{equation*}
    \lVert\bigwedge_{j=1}^{2g}\de z_j\rVert_{\Fal}^2(\tau_0)=\frac{1}{(2\pi)^{2g}}|\int_{\mathcal{A}_{\tau_0}}\bigwedge_{j=1}^{2g}\de z_j \bigwedge_{j=1}^{2g}\de \bar{z}_j|
    =\frac{1}{\pi^{2g}}\vol(\mathbb{C}^{2g}/\mathcal{O}_B(\tau_0,1)^t),
\end{equation*}
where the volume is defined from Lebesgue measure.

First, we claim that 
\begin{equation*}
\vol(\mathbb{C}^{2g}/\mathcal{O}_B(\tau_0,1)^t)=\prod_{i=1}^g((\Img(\tau_{0,i}))^{2}\vol(M_2(\mathbb{R})/\sigma_i(\mathcal{O}_B))),
\end{equation*}
where we endow $M_2(\mathbb{R})$ the Lebesgue measure again. To prove this claim, it is sufficient to check it
block-by-block for each $i$, since Lebesgue measures' product is still Lebesgue measure. Then see the argument in \cite[page 23]{Yuan2}.

Next, we claim that 
\begin{equation*}
    \vol(M_2(\mathbb{R})/\sigma_i(\mathcal{O}_B))=\sigma_i (d_B)
\end{equation*}
for each $i=1,...,g$. This is because for any full lattice $\Lambda$ of $M_2(\mathbb{R})$, we have
\begin{equation*}
    \vol(M_2(\mathbb{R})/\Lambda)\vol(M_2(\mathbb{R})/\Lambda^{\#})=1,
\end{equation*}
where $\Lambda^{\#}$ is the dual lattice under the standard pairing $(x,y)\mapsto\tr(xy^*)$. This is an elementary
linear algebra conclusion. By \cite[Lemma 15.6.17]{Voi}, we have 
\begin{equation*}
    \#(\sigma_i(\mathcal{O}_B)^\#/\sigma_i(\mathcal{O}_B))=\sigma_i(d_B^2),
\end{equation*}
which implies our claim.

Now we conclude that by taking product for all $i=1,...,g$, we have
\begin{equation*}
    \lVert\bigwedge_{j=1}^{2g}\de z_j\rVert_{\Fal}^2(\tau_0)
    =\frac{d_B^g}{\pi^{2g}}\prod_{i=1}^g(\Img(\tau_{0,i}))^2.
\end{equation*}
As $\tau_0\in\mathcal{H}^g$ varies, this implies
\begin{equation*}
    \lVert\bigwedge_{j=1}^{2g}\de z_j\rVert_{\Fal}^2
    =\frac{d_B^g}{\pi^{2g}}\prod_{i=1}^g(\Img(\tau_{i}))^2.
\end{equation*}
Apply the explicit formula in Theorem \ref{computation2}, we have
\begin{equation*}
    \lVert\bigwedge_{j=1}^{2g}\de z_j\rVert_{\Fal}^2=\lVert\psi(\bigwedge_{j=1}^{2g}\de z_j)\rVert_{\Pet}^2.
\end{equation*}
Thus, we complete the proof of our main theorems.

\

\noindent \small{School of mathematical sciences, Peking University, Beijing 100871, China}

\noindent \small{\it Email: ziqiguo0603@pku.edu.cn}

\end{document}